\documentclass[a4paper,12pt,oneside,reqno]{amsart}

\usepackage[utf8]{inputenc}
\usepackage[colorlinks]{hyperref}
\usepackage{comment}
\usepackage{lmodern}
\usepackage{geometry}
\usepackage{array}
\usepackage{graphicx}
\graphicspath{{./pic/}}
\usepackage{amssymb, amsfonts, amsthm, amsmath}
\usepackage{enumerate}
\usepackage{enumitem}
\usepackage[english]{babel}
\usepackage[font=footnotesize]{caption}
\usepackage{xr}
\usepackage{txfonts,amssymb,amsmath,amsthm,bbm}

\newtheorem{theorem}{Theorem}[section]
\newtheorem{corollary}[theorem]{Corollary}
\newtheorem{lem}[theorem]{Lemma}
\newtheorem{prop}[theorem]{Proposition}
\newtheorem{thm}[theorem]{Theorem}
\newtheorem{rem}[theorem]{Remark}
\newtheorem{dfn}[theorem]{Definition}

\numberwithin{equation}{section}

\newcommand{\bbE}{\mathbb{E}}
\newcommand{\bbV}{\mathbb{V}}

\newcommand{\bbP}{\mathbb{P}}
\newcommand{\bbT}{\mathbb{T}}
\newcommand{\bbN}{\mathbb{N}}

\newcommand{\bbR}{\mathbb{R}}

\newcommand{\eps}{\epsilon}

\newcommand{\ul}{\underline}
\newcommand{\wh}{\widehat}

\newcommand{\cC}{\mathcal C}

\newcommand{\cB}{\mathcal B}

\newcommand{\cT}{\mathcal T}

\newcommand{\rmc}{{\rm c}}

\newcommand{\rmd}{{\rm d}}
\newcommand{\rme}{{\rm e}}

\newcommand{\rmP}{{\rm P}}

\newcommand{\rmE}{{\rm E}}
\newcommand{\rmZ}{{\rm Z}}

\DeclareMathOperator*{\argmin}{arg\,min}
\DeclareMathOperator*{\argmax}{arg\,max}

\newcommand{\abs}[1]{\left| #1\right|}

\newcommand{\calB}{\mathcal{B}}
\newcommand{\calC}{\mathcal{C}}
\newcommand{\calD}{\mathcal{D}}
\newcommand{\calE}{\mathcal{E}}
\newcommand{\calF}{\mathcal{F}}

\newcommand{\calT}{\mathcal{T}}

\newcommand{\fra}{\mathfrak{a}}

\newcommand{\bbG}{\mathbb{G}}

\newcommand{\bbL}{\mathbb{L}}
\newcommand{\bbW}{\mathbb{W}}

\newcommand{\sfl}{{\mathsf l}}

\newcommand{\sfr}{{\mathsf r}}

\newcommand{\sfC}{\mathsf{C}}

\newcommand{\sfP}{\mathsf{P}}

\newcommand{\sfR}{\mathsf{R}}

\newcommand{\sfZ}{\mathsf{Z}}

\newcommand{\ub}{\underline{b}}

\newcommand{\lb}{\left(}
\newcommand{\rb}{\right)}
\newcommand{\lbr}{\left\{}
\newcommand{\rbr}{\right\}}

\newcommand{\dd}{{\rm d}}

\newcommand{\case}[1]{C{\small ASE}\,#1.}

\def\1{\ifmmode {1\hskip -3pt \rm{I}}
\else {\hbox {$1\hskip -3pt \rm{I}$}}\fi} 

\newcommand{\be}[1]{\begin{equation}\label{#1}}
\newcommand{\ee}{\end{equation}}

\newcommand{\ent}[1]{\mathrm{ent}\lb #1\rb}

\newcommand{\red}{\textcolor{red}}

\newcommand{\FromItamar}[1]{\footnote{\red {Itamar: #1}}}

\begin{document}

\title[Percolation with clustering]
{Wetting Transition on Trees I: Percolation With Clustering}
\author[A. Cortines]{Aser Cortines}
\address{A. Cortines}
\email{aserpeixoto@gmail.com}
\author[I. Harel]{Itamar Harel}
\address{I. Harel,
	Technion - Israel Institute of Technology. Haifa, 3200003,
	Israel.}
\email{itamarharel01@gmail.com}
\author[D. Ioffe]{Dmitry Ioffe}
\address{D. Ioffe, Deceased.}
\author[O. Louidor]{Oren Louidor}
\address{O. Louidor,
	Technion - Israel Institute of Technology. Haifa, 3200003,
	Israel.}
\email{oren.louidor@gmail.com}

\begin{abstract}
	A new ``Percolation with Clustering'' (PWC) model is introduced, where (the probabilities of) site percolation configurations on the leaf set of a binary tree are rewarded exponentially according to a generic function, which measures the degree of clustering in the configuration. Conditions on such ``clustering function'' are given for the existence of a limiting free energy and a wetting transition, namely the existence of a non-trivial percolation parameter threshold above and only above which the set of ``dry'' (open) sites have an asymptotic density. Several examples of clustering functions are given and studied using the general theory. The results here will be used in a sequel paper to study the wetting transition for the discrete Gaussian free field on the tree subject to a hard wall constraint.
\end{abstract}

\maketitle

\begin{flushright}
\begin{minipage}{0.5\textwidth}
\small
{\it Dedicated to the memory of our dear friend and mentor Dima Ioffe, who was the driving force behind this work. The void left by his sudden passing will remain forever.  }
\end{minipage}
\end{flushright}
\medskip\

\section{Introduction and results}
\subsection{Introduction, context and motivation}
Site Percolation on a finite graph $\bbG = (\bbV, \bbE)$ is a probability law on the set of configurations $\{\sigma :\: \sigma \in \{0,1\}^{\bbV}\}$, under-which $\{\sigma(v) :\: v \in \bbV\}$ are i.i.d Bernoulli random variables with a prescribed parameter $p \in (0,1)$. For reasons which will become clear shortly, we shall refer to a site $v$ where $\sigma(v)=1$ as dry and otherwise wet. Identifying a configuration $\sigma$ with the dry-set: $A = \{v \in \bbV :\: \sigma(v) = 1\}$, and letting $p = (1+\rme^{-J})^{-1}$, this law can be written as
\begin{equation}
\label{e:Z1}
\sfP_{\bbV}^{J}(A) := \frac{1}{Z^{J}_\bbV}\,{\rm e}^{J \abs{A}}
\quad ; \qquad  A \subseteq \bbV \,,
\end{equation}
where $|A|$ denotes the cardinality of $A$, and $Z^J_V$ is a normalizing constant.

The usual question in this model is that of percolation, namely the existence with high probability (under the percolation measure and at large graph volume) of a macroscopic connected component of dry sites.  In this manuscript, we make two changes to this model. First, we drop the connectivity requirement and study the scale of the dry-set in its entirety. Second, we consider a modified version of~\eqref{e:Z1}, in which the probability of a configuration is reweighed exponentially according to some measure of the level of clustering in the corresponding dry-set. 

For a formal definition, let $\Phi: \{0,1\}^{\bbV} \to \bbR$ be a function on configuration, to be thought of as assigning smaller to more clustered sets, and consider the law
\begin{equation}
\label{e:Z2}
\sfP_\bbV^{\Phi, J}(A) := \frac{1}{Z_{\bbV}^{\Phi, J}}\,{\rm e}^{J \abs{A} - \Phi (A )}
\quad ; \qquad  A \subseteq \bbV \,,
\end{equation}
with $Z_\bbV^{\Phi,J}$ modified accordingly. We shall refer to the law~\eqref{e:Z2} as Percolation With Clustering (PWC). It will be shown that for proper choices $\Phi$ and $J$, a non-trivial competition between entropy (the number of configurations) and clustering-based energy (the exponent in~\eqref{e:Z2}) will occur, so that properties of the dry-set and in particular the question of its scale will become non-trivial. 

Aside from an intrinsic interest, motivation for this model can be found in the theory of localization and delocalization of random surfaces, which are subject to attractive and repulsive forces (see ~\cite{bolthausen2003localization, funaki, giacomin2001aspects, velenik2006localization} for excellent surveys on the subject). Concretely, consider the random field $h = (h_x :\: x \in \bbV)$, whose law $\rmP^f_\bbV$ has a density $f: \bbR^{\bbV} \to \bbR_+$ with respect to the Lebesgue measure on $\bbR^{\bbV}$. Attraction to the underlying domain can  be introduced be adding an atom at $0$ with mass $e^J$ to the each coordinate of the reference measure, where $J \in (-\infty, \infty)$ marks the magnitude of the  pinning to $0$. Repulsion can be additionally (or alternatively) imposed by conditioning the surface to be positive. This will force the surface to repel from the underlying domain, or otherwise suffer an entropic cost. Altogether one obtains the law,
\begin{equation}
\label{e:Z4}
\rmP^{f, +,J}_{\bbV} (\rmd h) := \frac{1}{\rmZ^{f, +,J}_\bbV}
 	f(h) \prod_{x \in \bbV} \Big( \rmd h_x + {\rm e}^{J} \delta_0 (\rmd h_x) \Big) 
 \quad ; \qquad h \geq 0 \,,
\end{equation}
with $\rmZ^{+,J}_\bbV$ the usual normalizing constant.

A common interpretation of the above law is that of the interface between gas and liquid phases, which coexists in some thermodynamical system, with the gas lying on top of the liquid, which in turn lyes above an impenetrable hard wall at level zero. The gas is attracted to the domain by gravitation, but cannot cross it because of the wall. Sites $x \in \bbV$ where $h_x = 0$, capture ``dry'' locations in the domain, where gas reaches all the way to the bottom, thereby preventing the liquid from wetting it. Taking this view, positivity (vanishing) of the density of dry sites, means localization (delocalization) of the interface, and thus a  partial (complete) wetting of the domain. Of particular interest is the situation where  wetting is partial if $J > J^*$ and complete if $J < J^*$, where $J^* \in (-\infty, \infty)$ is a finite critical pinning force.  If this is the case, the system is said to exhibit a wetting transition.

To study the law of the dry-set, it is natural to expand the expression on the right hand side of~\eqref{e:Z4} as,
\begin{equation}
\label{e:Z1.4}
\rmP_\bbV^{f,+,J} (\rmd h) := \sum_{A \subseteq \bbV}
\rme^{J|A|} \frac{Z_\bbV^{f,+,A}}{\rmZ_\bbV^{f,+,J}}\, \rmP_\bbV^{f,+, A}(\rmd h) \,,
\end{equation}
where 
\begin{equation}
 	\rmP^{f,+,A}_\bbV(\rmd h) = \frac{f(h)}{Z_\bbV^{f,+,A}} \prod_{x \in A^\rmc} \rmd h_x \prod_{x \in A} \delta_0(\rmd h_x)\, =\, \rmP_\bbV^f \big(\rmd h\,\big|\,h \geq 0,\, \,h_{|_A} \equiv 0\big)
\end{equation}
and
\begin{equation}
\label{e:Z1.6}
 	\rmZ_\bbV^{f,+,A} :=  \int_{\bbR^{\bbV}_+} f(h) \prod_{x \in A^\rmc} \rmd h_x \prod_{x \in A} \delta_0(\rmd h_x) \ = \ \rmP_\bbV^f \big(h \geq 0,\, \,h_{|_A} \equiv 0\big) \,.
 \end{equation}
The last probability, which should be interpreted in the sense of densities, captures the probability that the field is positive yet pinned to the domain at $A$. The law of the interface is therefore of a convex combination of the measures $\rmP^{f,+,A}_\bbV(\rmd h)$ with weights given by $\rme^{J|A|} \rmZ_\bbV^{f,+,A}/\rmZ_\bbV^{f,+,J}$. 

Focusing on the distribution of the dry-set, law~\eqref{e:Z1.4} implies
\begin{equation}
\rmP_\bbV^{f,+,J} \big(A) \equiv 
	\rmP_\bbV^{f,+,J} \big(\{0\}^A \times (\bbR \setminus \{0\})^{A^\rmc}\big)
	 = \frac{1}{\rmZ_\bbV^{f,J}} \rme^{J|A|} \rmZ_\bbV^{f,+,A} \,.
\end{equation}
This shows that the configuration of dry sites under $\rmP_\bbV^{f,+,J}$ follows a PWC law~\eqref{e:Z2} with clustering function
\begin{equation}
\label{e:Z1.27}
	\Phi(A) := -\ln \rmZ_\bbV^{f,+,A} \,.
\end{equation}
For many choices of $f$, the probability density on the right hand side of~\eqref{e:Z1.6} is positively related to the degree of clustering in $A$, as it is easier for the surface to be pinned to a set which is, in some sense, more clustered. Thus $\Phi$ measures clustering in the way designated.
A general theory for PWC measures will thus allow us to decide on the nature of wetting and the existence of a wetting transition for such models.

In what follows, we develop a general theory for PWC measures, where the underlying graph is (the set of leaves of) the binary tree. We give formal definitions for the notions mentioned above, such as a clustering function and the scale of the dry-set. We then include general results about the existence of the free energy, asymptotic positivity of the density of dry sites, and the existence of a wetting transition. We then provide several examples of PWC measures with clustering functions, which either arise from other statistical models, as in the case above, or are natural to consider, from a mathematical point of view. Examples of the s	econd kind are also chosen because they are the same time approximations of ``real'' PWC measures (such as the one introduced above) and also amenable to analysis. Results from the general theory are then specialized to the examples, yielding more concrete conditions which determine the nature of wetting, in the cases studied.

\subsection{General theory}
\label{s:1.1}
Let $\bbT^{(n)} := \bigcup_{i=0}^n 
\{0,1\}^i$ be the binary tree of depth $n$ rooted at $\mathbf{0}^{(n)}$. While we keep 
the genealogical relation as usual, for notational convenience it will be useful to the 
define the {\em age} (occasionally {\em level} or {\em generation}) $|v|$ of vertex (occasionally site) $v \in \bbT^{(n)}$, as its distance from the leaves, not the root as is often the case. 
As such, the root (at age $n$) is the oldest vertex, while the leaves (at age $0$) are the youngest, and the children of a vertex are one generation younger than their parent. We let $\bbL_i^{(n)}$ denote the set of vertices whose age is $i \in [0,n]$, and for $i \geq |v|$, write $v|_i$ for the ancestor of $v$ in $\bbL_i^{(n)}$, The super-index $(n)$ will be occasionally omitted if no confusion is created this way.

Given $n$ and a non-empty subset $A\subseteq \bbL_0^{(n)}$ of the leaves, we let $\mathcal{T}(A)$ be the sub-tree of $\bbT^{(n)}$ which is {\em generated} by $A$. That is, $\calT \lb A \rb$ is the subtree of $\bbT^{(n)}$ which includes all vertices $\{v|_\ell :\:  v \in A, \ell=0, \dots, n \}$ and all interconnecting edges which are inherited from $\bbT^{(n)}$. We shall say that two sets $A, A' \subseteq \bbL_0^{(n)}$ are {\em equivalent}, and write $A \sim A'$, if the corresponding trees $\mathcal{T}(A)$ and $\mathcal{T}(A')$ are graph isomorphic. This is clearly an equivalence relation.

The PWC measure is a site-percolation-type measure on the set of leaves $\bbL_0^{(n)}$. As such, it is a distribution over the set of all assignments of ``open'' and ``closed'' to each leaf in $\bbL_0^{(n)}$, or equivalently a distribution over subsets $A$ (of open sites) thereof. As before, we shall refer to open and closed sites as {\em dry} and {\em wet} respectively and to the set of dry (open) sites as the {\em dry-set}.
  The main ingredient in the definition of this measure is a {\em clustering function}, from which the only requirement is
\begin{dfn} 
 \label{dfn:CF} 
 A function $\Phi: 2^{\bbL_0^{(n)}} \to \bbR$ is called {\em clustering} if 
 $\Phi (\emptyset )= 0$ and 
 $\Phi (A ) = \Phi (B)$
 whenever $A\sim B$. 
\end{dfn}

With the above at hand, we can thus define:
\begin{dfn}
 \label{dfn:PWC} 
Let $n \geq 1$, $\Phi: 2^{\bbL_0^{(n)}} \to \bbR$ clustering, and $J \in \bbR$.
The PWC-measure with clustering function $\Phi$ and pinning force $J$ is a probability measure on $2^{\bbL_0^{(n)}}$ which is determined by
\begin{equation}
\sfP_n^{\Phi, J}(A) := \frac{1}{\sfZ_n^{\Phi, J}}\,{\rm e}^{J \abs{A} - \Phi (A )}
\quad ; \qquad  A \subseteq \bbL_0^{(n)} \,,
\end{equation}
where 
\begin{equation}
\sfZ_n^{\Phi, J} := \sum_{A \subseteq \bbL_0^{(n)}}{\rm e}^{J \abs{A} - \Phi (A)} 
\end{equation}
is the associated {\em partition function}.
\end{dfn}
\noindent
We also denote by $\rmE_n^{\Phi,J}$ the associated expectation. With a bit of abuse, we shall also denote by $A$ the canonical random variable on $2^{\bbL_0^{(n)}}$, so that, e.g. $\rmP_n^{\Phi, J}(u \in A)$ or $\rmE_n^{\Phi,J}|A|$ are well defined.

When $\Phi \equiv 0$, one recovers the usual Bernoulli site percolation on $\bbL_0^{(n)}$ with sites being independently open (dry) with probability $p(J) = (1+{\rm e}^{-J})^{-1}$. One should think of $\Phi(A)$ as a function which measures, somehow, the degree of clustering in the set $A$, so that lower value of $\Phi(A)$ means higher degree of clustering in $A$. Then the law $\sfP_n^{\Phi, J}$ can be seen as a modification of the usual percolation distribution in which configurations with a larger degree of clustering are rewarded exponentially. 

Note that, contrary to the intuition mentioned earlier, the definition of a clustering function does not inherently relate $\Phi(A)$ to any specific measure of clustering within $A$. While this provides more generality, some results will require such a relationship, which will be achieved by imposing additional constraints on $\Phi$.

The underlying phenomenological question in this work is whether, given a pinning force and a clustering function, the dry-set is typically ``microscopically'' or ``macroscopically'' large under the PWC measures $\rmP_n^{\Phi, J}$. 
Addressing this question necessitates, of course, a formal definition of the above notions of scale.
Taking the Physics angle, we fix $J \in \bbR$ and consider, naturally, not one, but a sequence of clustering function $\{\Phi_n\}$ indexed by the height of the tree. We then define the {\em free energy} for the corresponding sequence of PWC measures as the limit
\begin{equation}
\zeta(J) := \lim_{n \to \infty} \zeta^{\Phi_n}_n(J)
\end{equation}
where
\begin{equation}
\label{e:1.4}
\zeta^{\Phi}_n(J) := \frac{1}{2^n} \ln \sfZ_n^{\Phi, J} \,,
\end{equation}
with $\sfZ_n^{\Phi, J}$ being the partition function from Definition~\eqref{dfn:PWC}. While the limit need not exist in general, sufficient conditions under which it does are given below.

Assuming the existence of the limit, we can make a formal sense of the scale of the dry-set via the following definition, which is standard in the theory of random surfaces. Again, we adhere to the terminology in the latter context.
\begin{dfn}
\label{d:1.3}
Let $J \in \bbR$, and for each $n$, let $\Phi_n$ be a clustering function on $\bbT_n$. The sequence of PWC-measures $\{\rmP_n^{J, \Phi_n}\}$ exhibits {\em complete wetting} (a macroscopic dry-set) if $\zeta(J) = 0$ and otherwise {\em partial wetting} (a microscopic dry-set).
\end{dfn}

The standard motivation for the above definition is as follows. Recall that by definition $\Phi_n(\emptyset) = 0$, so that $\zeta^{\Phi_n}_n(J)$ is non-negative. If we assume further that $|\Phi_n(A)|$ grows at most linearly in $|A|$ uniformly in $n$ (which is natural and will be a consequence of our assumptions in the examples below), then the contribution to the partition function from dry-sets whose size is sublinear (in $2^n$) is also sublinear on an exponential scale. Consequently, if $\zeta^{\Phi_n}_n(J)$ is asymptotically positive it must be dominated by dry-sets having positive density, i.e. a macroscopic one.

As $\zeta(J)$ is evidently non-decreasing in $J$, it makes sense to state,
\begin{dfn}
\label{d:1.4}
The {\em critical pinning force} $J^*$ associated with the sequence of clustering functions $\{\Phi_n\}$ is
\be{eq:J-star-PC}
 J^* = \inf\lbr J ~:~ \zeta (J ) >0\rbr .
\ee
We shall say that $\{\Phi_n\}$ exhibits a {\em wetting transition} if  $J^* \in (-\infty, \infty)$.
\end{dfn}

Next, we wish to impose restrictions on $\{\Phi_n\}$ under which the above definitions are proper, namely the free energy exists. These will come in the form of monotonicity with respect to some clustering order on configurations. We shall then derive a necessary and sufficient condition for the existence of a wetting transition. Accordingly, with $u \wedge v$ denoting, customarily, the youngest common ancestor of $u,v \in \bbT_n$, we have,
\begin{dfn} 
 \label{dfn:MoreC}
 Given two subsets $A,B\subseteq \bbL_0^{(n)}$ of the same cardinality $\abs{A}= \abs{B}$, we say 
 that $A$ is {\em more clustered} than $B$ and write $A\prec B$, if there exists a bijection $\sigma :A\to B$, such 
 that 
 \be{eq:ApreqB} 
 |u \wedge v |\leq |\sigma_u \wedge  \sigma_v| \quad \forall\, u,v\in A .
 \ee
\end{dfn}
It is easy to see that above relation is a partial order on $2^{\bbL_0{(n)}}$. We can then also have,
\begin{dfn} 
 \label{dfn:MC} 
 A clustering function $\Phi$ is called {\em monotone} if \\
 (a) For any $A$ and $B$ of equal cardinality 
 \be{eq:MC-a} 
 A\prec B\ \Rightarrow\ \Phi (A ) \leq \Phi (B) .
 \ee
 (b) For any disjoint subsets $A,B\subseteq \bbL_0$ , 
 \be{eq:MC-b}
 \Phi (A\cup B )\leq \Phi (A) + \Phi (B) .
 \ee
\end{dfn}

Now, for a clustering function $\Phi$ on $\bbL_0^{(n)}$ and $a_0 \in [0, 2^n]$, we define the {\em canonical partition function} associated with $\Phi$ as 
\be{eq:Phi-Gpf} 
\bbW_n (a_0) \equiv \bbW^{\Phi}_n (a_0 ) :=  \sum_{\abs{A} = a_0} {\rm e}^{-\Phi (A)} \,,
\ee
so that evidently 
\be{eq:Phi-J} 
\sfZ^{\Phi,J}_n = \sum_{a_0 = 0}^{2^n} {\rm e}^{J a_0} \bbW^\Phi_n (a_0 ) \,,
\ee
Then given a sequence $\{\Phi_n\}$ and $\epsilon \in [0,1]$ dyadic, namely a real number which is equal to $k/2^m$ for some integers $k$ and $m$, we also define the {\em canonical free energy} as 
\be{eq:Phi-FE} 
\omega(\epsilon) := \lim_{n\to\infty} \omega_n^{\Phi_n} (\epsilon) \,.
\ee
where
\be{eq:Phi-FE1} 
\omega^{\Phi}_n(\epsilon) := \frac{1}{2^n}\ln 
\bbW_n^{\Phi} (\epsilon 2^n ) \,,
\ee
whenever the limit exist.

The next proposition gives a sufficient condition for the existence of the above limit for all dyadics and, consequently, the existence of $\omega$ on all of $[0,1]$. To this end, let us denote by $\bbT^{(n)}_{\sfr}$ and $\bbT^{(n)}_{\sfl}$ the two connected components of $\bbT^{(n)} \setminus {\bf 0}^{(n)}$. 
These are clearly isomorphic to $\bbT^{(n-1)}$ and so whenever $A$ is a subset of $\bbL^{(n)}_{0,\sfr}$ or $\bbL^{(n)}_{0,\sfl}$, we let $\Phi_{n-1} (A)$ be the value of $\Phi_{n-1}$ on the image of $A$ under such isomorphism.

\begin{prop}
 \label{prop:Phi-FE-A} 
Suppose that for all $n \geq 1$, the clustering function $\Phi_n$ is monotone and non-negative 
and that there exists a non-negative sequence $\lbr \gamma_n\rbr$, satisfying 
\begin{equation}
	\sum_n \gamma_n 2^{-n} <\infty \,,
\end{equation}
such that for all $n \geq 1$ and all $A \subseteq \bbL_{0, \sfr /\sfl}^{(n)}$,
\be{eq:Phi-A} 
\Phi_n (A ) \leq \Phi_{n-1} (A ) + \gamma_n \,.
\ee
Then the canonical free energy $\omega$ is well defined, finite and concave on all dyadic $\epsilon \in [0,1]$ and therefore can be uniquely continuously extended to all of $[0,1]$.
\end{prop}
Henceforth, we will identify the canonical free energy $\omega$ with its continuous extension via the above proposition, whenever its conditions are satisfied. We can now address the existence of the (non-canonical) free energy.
\begin{thm} 
\label{thm:Phi-wet} 
Let $\{\Phi_n\}$ be as in Proposition~\ref{prop:Phi-FE-A}. Then the free energy $\zeta$ is well defined, finite and convex on $(-\infty, \infty)$.
 Furthermore, it is related to the canonical free energy $\omega$ as follows: 
 \be{eq:Phi-psif}
 \zeta (J) = \max_{\epsilon \in [0,1]}\lbr J\epsilon + \omega (\eps )\rbr . 
 \ee
 \end{thm}
 A consequence of the above representation is the following characterization of the critical pinning force for the sequence $\{\Phi_n\}$.
\begin{corollary}
\label{c:1.9}
Let $\{\Phi_n\}$ be as in Proposition~\ref{prop:Phi-FE-A}. Then the associated critical pinning force is given by
\be{eq:Phi-f-J} 
J^* = - \frac{\dd^+}{\dd\eps } \omega(0 ) = - \lim_{\eps\downarrow 0} 
\frac{\omega  (\eps )}{\eps}
= - \lim_{\substack\eps\downarrow 0} 
\lim_{n\to\infty} \frac{1}{\epsilon 2^n}\ln 
\bbW_n (\epsilon 2^n ) \,.
\ee
In particular, there is a wetting transition if and only if the right derivative of $\omega$ at $0$ is finite.
\end{corollary}

We note that another conventional way to characterize the scale of the dry-set is via the more probabilistic quantity (below $u \in \bbL_0^{(n)}$ is any leaf),
\begin{equation}
	\rho_n(J) =  \rho^{\Phi}_n(J) := \frac{1}{2^n} \rmE_n^{\Phi, J} |A| = \rmP_n^{\Phi,J}(u \in A)\in [0,1] \,,
\end{equation}
which measures the density of the dry-set. The corresponding asymptotic quantity for a sequence of clustering functions $\{\Phi_n\}$ is given by
\begin{equation}
\label{e:1.18}
	\rho(J) := \lim_{n \to \infty} \rho_n^{\Phi_n}(J)  \,,
\end{equation}
assuming again that the limit exists (otherwise, one may take the limit inferior). Complete wetting can then be defined as the regime when $\rho(J) = 0$, with partial wetting occurring when the latter is positive. 

The link between the two definitions of the wetting regimes is via the relation, 
\begin{equation}
\label{e:1.19}
\zeta_n^\Phi(J) = \int_{-\infty}^J \rho_n^\Phi(J') \rmd J' \,,
\end{equation}
which can be easily verified by differentiating~\eqref{e:1.4}. Nevertheless, it is not clear that one can pass to the limit in the above relation. This is resolved by
\begin{corollary}
\label{c:1.10}
	Let $\{\Phi_n\}$ be as in Proposition~\ref{prop:Phi-FE-A}. Then for all $J$ such that the maximum in~\eqref{eq:Phi-psif} is attained uniquely, the limit in~\eqref{e:1.18} exists and
\begin{equation}
	\rho(J) = \frac{\rmd}{\rmd J} \zeta(J) = \argmax_{\epsilon \in [0,1]}\lbr J\epsilon + \omega (\eps )\rbr . 
\end{equation}
Moreover uniqueness in~\eqref{eq:Phi-psif} holds except for at most countably many values of $J$.
\end{corollary}

\subsection{Examples}
\label{s:1.2}
Let us now discuss several examples of PWC measures. The first two examples are analogous to the example in Subsection~\ref{s:1.1}, which was stated in a more general context.
\subsubsection{Pinned random surfaces}
\label{s:Z1.2.1}
For $n \geq 1$, let $f: \bbR^{\bbL^{(n)}_0} \to \bbR_+$ be a density function which is invariant under (the restriction to $\bbL_0^{(n)}$ of) all isomorphisms of $\bbT_n$.
Denote by $\rmP^f_n$ be the probability measure on $\bbR^{\bbL_0^{(n)}}$ whose density w.r.t. the Lebesgue measure is $f$. One should think of $\rmP^f_n$ as being the law of a random surface $h$ indexed by the leaves, so that $h(u)$ is the height at $u \in \bbL_0^{(n)}$. 

For $J \in \bbR$ we introduce {\em delta pinning} at force $J$ by considering the modified law
\begin{equation}
\label{e:O1}
\rmP_n^{f,J} (\rmd h) := \frac{1}{\rmZ_n^{f,J}}
 	f(h) \prod_{x \in \bbL_0^{(n)}} \big( \rmd h_x + {\rm e}^{J} \delta_0 (\rmd h_x) \big) \,,
\end{equation}
with $\rmZ_n^{f,J}$ chosen as to make the above a probability measure. Then, as in~\eqref{e:Z1.4}, expanding the expression on the right hand side of~\eqref{e:O1}, one may rewrite the pinned law $\rmP_n^{f,J}$ as
\begin{equation}
\rmP_n^{f,J} (\rmd h) := \sum_{A \subseteq \bbL_0^{(n)}}
\rme^{J|A|} \frac{Z_n^{f,A}}{\rmZ_n^{f,J}}\, \rmP_n^{f,A}(\rmd h) \,,
\end{equation}
where 
 \begin{equation}
 	\rmP^{f,A}_n(\rmd h) = \frac{f(h)}{Z_n^{f,A}} \prod_{x \in A^\rmc} \rmd h_x \prod_{x \in A} \delta_0(\rmd h_x) 
 	\; , \quad
 	\rmZ_n^{f,A} :=  \int f(h) \prod_{x \in A^\rmc} \rmd h_x \prod_{x \in A} \delta_0(\rmd h_x) \,.
 \end{equation}
In particular,
\begin{equation}
\label{e:1.5}
	\rmP_n^{f,J} \big(\{0\}^A \times (\bbR \setminus \{0\})^{A^\rmc}\big)
	 = \frac{1}{\rmZ_n^{f,J}} \rme^{J|A|} \rmZ_n^{f,A} \,,
\end{equation}
so that the configuration of pinned sites under $\rmP_n^J$ is a PWC measure with clustering function
\begin{equation}
\label{e:1.27}
	\Phi(A) = \Phi^f_n(A) := -\ln (\rmZ_n^{f,A}/\rmZ_n^{f,\emptyset}) 
	\quad ; \qquad A \subseteq \bbL_0^{(n)} \,,
\end{equation}
and pinning force $J$.

A particular case where such field $h$ arises, is when one considers the restriction to $\bbL_n^{(0)}$ of a Gibbs gradient field $h'$ on $\bbT_n$, whose law is given by 
\begin{equation}
\label{e:1.26}
\rmP^V_n(\rmd h) = 
	\frac{1}{\rmZ^V_n} \exp \Big\{ -\sum_{x \sim y} V \big( h_x -  h_y \big) \Big\}
	\prod_{x \in \bbT_n \setminus \{0\}} \rmd  h_x\, \delta_0 (\rmd h_{\bf 0}) \,,
\end{equation}
where the sum is over all neighboring vertices in $\bbT_n$ and $V: \bbR \to \bbR_+$ is a smooth, symmetric and convex function.
When $V(s) = s^2$, the field $h'$ is the discrete Gaussian free field (DGFF) on $\bbT_n$.
For this particular field, it is known the law of pinned sites under the corresponding measure $\rmP^{J,f}_n$ is strong FKG~\cite[Lemma~2.1]{bolthausen2001critical}. 
This shows that the corresponding clustering function, we shall denote in this case by
\begin{equation}
\Phi^V_n \equiv \Phi^f_n \,,
\end{equation}
satisfies the second condition in Definition~\ref{dfn:MC} of monotonicity. However the validity of the first condition is not clear.

Nevertheless, for a general $V$, it can be shown by domination over Bernoulli site percolation (cf.~\cite[Theorem 2.4]{bolthausen2001critical}) that $\rho^\Phi_n(J)$ is uniformly positive for all $J > -\infty$. 
Since the free energy $\zeta$ exists by the strong FKG property of $\rmP^V_n$ (c.f.,~\cite[Lemma 7.8]{funaki}), we may use~\eqref{e:1.19} to conclude that there is no wetting transition in the sense of Definition~\ref{d:1.4}, for any $V$ as above.

\subsubsection{Pinned random surface above a hard wall}
\label{s:1.1.2}

Continuing from the previous example, an interaction of the field $h$  with a, so-called, {\em hard wall} placed above $\bbL_0^{(n)}$ at height zero can be modeled by considering the conditional measure
\begin{equation}
\label{e:1.28}
\rmP_n^{f,+}(\rmd h) := \rmP_n^{f,+}(\rmd h \,|\, \Omega^+)
= \frac{\rmP_n^{f,+}(\rmd h)1_{\Omega^+_n}(h)}{\rmP_n^{f,+}(\Omega^+)}
\quad , \qquad
\Omega^+_n := \bbR_+^{\bbL_0^{(n)}}
\end{equation}
Here, a common consequence of entropy is that the field does not just stay above the wall, but is rather typically pushed away from it. Thus, conditioning introduces a {\em repulsion} force. 
An opposing attraction force can be added, by introducing delta pinning, as before. This leads to considering the law
\begin{equation}
\label{e:O1plus}
\rmP_n^{f,+,J} (\rmd h) := \frac{1}{\rmZ_n^{f,+,J}}
 	f(h) \prod_{x \in \bbL_0^{(n)}} \big( \rmd h_x + {\rm e}^{J} \delta_0 (\rmd h_x) \big) 
 \quad ; \qquad h \in \Omega^+_n
\end{equation}

As before, by expanding~\eqref{e:1.28} we can express the law of dry sites as
\begin{equation}
\label{e:1.29}
\rmP_n^{f,+,J} \big(\{0\}^A \times (\bbR \setminus \{0\})^{A^\rmc}\big)
	 = \frac{1}{\rmZ_n^{f,+,J}} \rme^{J|A|} \rmZ_n^{f,+, A}
\end{equation}
where
\begin{equation}
\rmZ_n^{f,+,A} :=  \int_{\Omega^+} f(h) \prod_{x \in A^\rmc} \rmd h_x \prod_{x \in A} \delta_0(\rmd h_x) \,,
\end{equation}
This law is again a PWC measure, this time with clustering function 
\begin{equation}
\label{e:1.27a}
	\Phi(A) = \Phi^{f,+}_n(A) := -\ln (\rmZ_n^{f,+,A}/\rmZ_n^{f,+,\emptyset})
	\quad ; \qquad A \subseteq \bbL_0^{(n)} \,,
\end{equation}
and pinning force $J$.

When $h$ is the restriction to $\bbL_0^{(n)}$ of a Gibbs-gradient field $h'$ on $\bbT_n$, of the form considered in the previous example, then the corresponding law $\rmP_n^{f,+}$ is strong FKG (c.f.,~\cite[Lemma~3.1]{deuschel1996entropic}). Since having a pinned site is a decreasing event, the corresponding PWC-measure~\eqref{e:1.29} is also FKG. In particular, this implies that the corresponding clustering function, which we shall denote in this case by
\begin{equation}
\label{e:1.32}
\Phi^{V,+}_n \equiv \Phi^{f,+}_n \,,
\end{equation}
 satisfies the second condition in Definition~\ref{dfn:MC} of monotonicity. Again, whether the first condition thereof holds is not clear.

Nevertheless because of FKG, both the free energy $\zeta$ and the limiting density of pinned sites $\rho$ exist (c.f.~\cite[Section 6]{velenik2006localization} and~\cite[Lemma 7.8]{funaki}). It can be shown, by adapting the proofs when the underlying graph is the two-dimension lattice, that whenever $V$ is Gaussian or Lipschitz that there is a wetting transition in the sense of Definition~\ref{d:1.4} (see,~\cite[Section 6]{velenik2006localization}). 

\subsubsection{The zero clustering function}
For a sanity check, let us quickly address the $\Phi \equiv 0$ case. As mentioned before, in this case $\rmP_n^{\Phi,J}$ is the standard Bernoulli site percolation with probability $p(J) = (1+{\rm e}^{-J})^{-1}$ for a site to be dry. In this case, the canonical free energy $\omega(\epsilon)$ is trivially given by the Binomial entropy function
\be{eq:Phi-ent}
\ent{\epsilon}: = \epsilon \ln \lb\frac{1}{\epsilon}\rb  + (1-\epsilon)\ln\lb \frac{1}{1-\epsilon}\rb
\ee
As the right derivative of $\ent{\epsilon}$ at $0$ is $\infty$, there is no wetting transition. This is in agreement with the (asymptotic) density of dry sites being $p(J) > 0$ for all $J > -\infty$.

\medskip

\subsubsection{Clustering via first order branching patterns}
\label{sss:2.3}
Next we study a clustering function which is not obtained as a reduction from another statistical model. In this case, the degree of clustering is determined based on the tree (hierarchal) distances between all leaves in the dry-set $A \subseteq \bbL_0^{(n)}$. More precisely, for each $k \in [1,n]$, let $b_k$ be the number of branching points at level $k$ of the tree. A branching point is a vertex in $\bbT^{(n)}$, which has at least two descendants in $A$, one through each of its direct children. Then the value of $\Phi(A)$ is determined as a linear combination of the branching numbers $(b_k)_{k=1}^n$, with prescribed coefficients $(h_k)_{k=1}^n$ acting as the parameters of the model.

Now, it is not difficult to see that the total number of branching points, namely the sum of all $b_k$-s, is $|A|-1$, so that for sets $A$ with a prescribed cardinality, the $\ell_1$-weight of the branching pattern $(b_k)_{k=1}^n$ is fixed. However, the more this weight is concentrated on $b_k$-s with smaller index $k$, the lower the level of the branching points for $A$ tend to be, resulting in a more clustered set. Therefore, if the parameters $(h_k)_{k=1}^n$ are chosen to be non-decreasing in $k$, then $\Phi(A)	$ will adhere to the intended convention of being smaller for more clustered sets. This is formalized in Proposition~\ref{p:1.6} below, which shows that such $\Phi$ is indeed a monotone clustering function.  Necessary and sufficient conditions on the parameters $(h_k)_{k=1}^n$ for the existence of the free energy and, more importantly, a wetting transition, are then given in Proposition~\ref{p:1.9} and Proposition~\ref{p:1.10} that follow.

We remark that aside from a natural way to quantify clustering, the advantage of using a linear function of the branching pattern to define $\Phi$, is that the conditions for montonicity (Definition~\ref{dfn:MC}), existence of the free energy (Proposition~\ref{prop:Phi-FE-A}) and a wetting transitions (Theorem~\ref{thm:Phi-wet}), from the general theory, could be reduced  to conditions on the parameters $(h_k)_{k=0}^n$ of the model at hand. The latter are much easier to check. Moreover, such results for this model and also for its generalized version in the next example, could be used together with comparison arguments to imply similar statements for PWC-measures which arise from random-surfaces -- of the kind in Subsections~\ref{s:Z1.2.1} and~\ref{s:1.1.2}. See Subsection~\ref{s:1.4} for a more thorough discussion of this.

Proceeding formally, the set of {\em branching points} corresponding to $A$ is defined as
 \be{eq:BP}
  \calB (A) = \lbr u\wedge v~\middle\vert~ u,v\in A\rbr \,.
 \ee
Observe that $A \subseteq \calB(A) \subseteq \cT(A)$. The restriction to branching points at level 
$k \in [0,n]$ is given by
\begin{equation}
\calB_{k}(A) = \calB (A) \cap\bbL_k \,,
\end{equation}
with $b_k := |\calB_{k}|$ denoting their number. The {\em first order branching pattern} corresponding to $A \subseteq \bbL_0^{(n)}$ is the sequence
\begin{equation}
	\ul{b}(A) := (b_k)_{k=0}^n \,.
\end{equation}
We can now state,
\begin{dfn}
\label{d:1.11}
Let $n \geq 1$ and $h$, $h_k$ for $k \in [0,n]$ be real numbers. 
Then the function $\Phi \equiv \Phi^{\{h_k\},h}_n: 2^{\bbL_0^{(n)}} \to \bbR$, defined by setting $\Phi(\emptyset) = 0$ and for each $A \subseteq \bbL_0^{(n)} \setminus \{\emptyset\}$ via 
\begin{align} 
\label{eq:BP-ex1} 
\Phi(A) := \sum_{k=0}^n h_k b_{k} + h 
\quad ; \qquad \ul{b}(A) = (b_k)_{k=0}^n \,,
\end{align}
is called a first order branching clustering function with parameters $(h_k)_{k=0}^n$ and $h$.
\end{dfn}

As, clearly, $\ul{b}(A) = \ul{b}(A')$ for $A \sim A'$, the function $\Phi$ is clustering.
The next proposition shows that it is also monotone, provided the sequence $(h_k)_{k \geq 1}$ is such.
\begin{prop}
\label{p:1.6}	
Suppose that $(h_k)_{k=0}^n$ is non-decreasing and that $h \geq h_n$, then the corresponding 
first order branching clustering function $\Phi$ is monotone.
\end{prop}

For the remainder of this sub-section we assume that $(h_k)_{k=0}^\infty$ is a given sequence of real numbers and for each $n \geq 1$, we let $\Phi_n \equiv \Phi_n^{\{h_k\}, h_n} $ be a first order branching clustering function with parameters $(h_k)_{k=0}^n$ and $h = h_n$. 
Turning to the question of wetting, we have

\begin{prop} 
\label{p:1.8}
Suppose that $(h_k)_{k=0}^\infty$ is non-decreasing. Then, the canonical free energy $\omega$ associated with $\{\Phi_n\}$ exists if and only if,
\be{eq:Lev1Exist} 
\sum_{k} 2^{-k} h_k  <\infty . 
\ee
\end{prop}

Next we give sufficient and necessary conditions on the sequence $\{h_k\}$ for the existence of a wetting transition. 
\begin{prop}
\label{p:1.9}
Suppose that $(h_k)_{k=0}^\infty$ is non-decreasing and that~\eqref{eq:Lev1Exist} holds. Then there is a wetting transition at 
$J^* \in \lb  2 \ln 2 + h_0 - \ln \kappa_1 ,\, \infty\rb$
if 
\begin{equation}
\label{e:kappa1}
	\kappa_1 := \sum_{k=1}^\infty 2^k\rme^{-h_k}  < \infty \,,
\end{equation}
\end{prop}

Turning to the necessary condition, henceforth
we denote the Laplace Transform of sequence $(v_n)_{v \geq 0}$ by $\hat{v}$. That is,
\begin{equation}
	\hat{v}(s) = \sum_{n=0}^\infty e^{-ns} v_n
	\ , \ \ s \geq 0 \,,
\end{equation}
whenever the series converges absolutely. We shall also write $v^+$ for the sequence $(v_n \vee 0)_{n \geq 0}$.

\begin{prop}
\label{p:1.10}
Assume that $(h_k)_{k=0}^\infty$ is non-decreasing and that~\eqref{eq:Lev1Exist} holds. Then there is no wetting transition if
\begin{equation}
\label{e:sufficient lim b1}
	\lim_{s \downarrow 0} \big(\ln \frac1s - s\wh{g^+}(s)\big) = \infty \,,
\end{equation}
where $(g_k)_{k=0}^\infty$ is defined via,
\begin{equation}
\label{e:gk b1}
g_k := h_k - (\ln 2) k  
\quad ,\quad   k \geq 0 \,.
\end{equation}
\end{prop}
An immediate consequence of the above is the next corollary which follows by a standard Tauberian argument. It gives an almost complementary condition to that in Proposition~\ref{p:1.9}.
\begin{corollary}
\label{c:1.11}
Assume that $(h_k)_{k=0}^\infty$ is a non-decreasing sequence and that~\eqref{eq:Lev1Exist} holds. Then there is no wetting transition if
\begin{equation}
\lim_{k \to \infty} k2^k\rme^{-h_k} = \infty \,,
\end{equation}
or equivalently, if
\begin{equation}
\label{eq:B1-wetting-necessary-general-main}
\lim_{k \to \infty} (\ln 2) k+ \ln k -h_k = \infty 
\,.
\end{equation}
\end{corollary}

\subsubsection{Clustering via second order branching patterns}
\label{s:1.2.4}
Next we generalize the previous example, by considering two levels of branching. Here $\Phi$ will be a linear function of the second-order branching numbers $b_{k,\ell}$ where $1 \leq \ell < k \leq n$. The value of $b_{k,\ell}$ is the number of branching points at level $k$, such that their youngest ancestor in the tree, which is also a branching point,  is at level $l$. This is a refined prism for the branching structure of the induced tree $\cT(A)$ of the dry-set $A$, and as such a refined measure of the clustering of $A$. As before, higher concentration of the $\ell_1$-weight of the second-order branching pattern $(h_{k,\ell})_{1 \leq \ell < k \leq n}$ on lower indices, implies more clustering.

In analog to the previous example, we show that monotonicity, with respect to lexicographic order of the indices, of the linear coefficients $(h_{k,\ell})_{1 \leq \ell < k \leq n}$ as the parameters in the definition of $\Phi$, is sufficient for its monotonicity as a clustering function. We then present necessary and sufficient conditions on these parameters for the existence of the free energy and a wetting transition. 

Formally, if $u,v \in \cB(A)$, we shall say that $v$ is a {\em direct branching ancestor} of $u$ if
$v$ is the youngest ancestor of $u$ which is also a branching point. In this case $u$ is a {\em direct branching descendant} of $v$. For $0 \le \ell < k \le n$, we then let
\begin{equation}
	\cB_{k,\ell}(A) := \big\{(u,v) \in \cB(A) :\: v \text{ direct branching ancestor of } u,
	\, |u| = \ell, |v| = k \big\} \,,
\end{equation}
and set 
\begin{equation}
b_{k,\ell} := 
\begin{cases}
	|\cB_{k,\ell}(A)| & 0 \leq \ell < k \leq n\,,\\
	\1_{\{\max_{\cB(A)} |v|\}}(\ell) & 0 \leq \ell < k = n+1 \,.
\end{cases}
\end{equation}
Thus, $b_{k,\ell}$ is the number of branching points at level $\ell$ whose direct branching ancestor is in level $k$, where the oldest branching point: $\argmax_{v \in \cB(A)} |v|$, is thought of as having a branching ancestor at level $n+1$.
The {\em second order branching pattern} corresponding to $A \subseteq \bbL_0^{(n)}$ is the triangular array:
\begin{equation}
	\ul{b}(A) = (b_{k, \ell} :\: 0 \le \ell < k \le n +1)\,,
\end{equation}
where we use the same notation from the previous example.
Observe that the first order branching pattern elements $b_k$ can be easily reconstructed from the second order pattern via,
\be{eq_one_two_bf}
b_i (A) = \sum_{k=i+1}^{n+1} b_{k,i}= \frac{1}{2}\sum_{\ell=0}^{i-1} b_{i,\ell} \,.
\ee

We can now define the second order branching clustering function.
\begin{dfn}
For $n \geq 1$, let $(h_{k,\ell} :\: 0 \leq \ell < k \leq n+1)$ and $h\in \bbR$ be real numbers.
Then, the function $\Phi \equiv \Phi^{\{h_{k,\ell}\}, h}_n: 2^{\bbL_0^{(n)}} \to \bbR$, defined by setting $\Phi(\emptyset) = 0$ and for each $A \subseteq \bbL_0^{(n)} \setminus \{\emptyset\}$ via 
\begin{align} \label{eq:BP-ex2} 
 \Phi (A ) := \sum_{0 \leq \ell < k \leq n+1} h_{k ,\ell} b_{k, \ell} (A) + h 
 \; ; \quad \ul{b}(A) = (b_{k,\ell})_{0 \leq \ell < k \leq n+1} \,,
\end{align}
is called a {\em second order branching clustering function} with parameters 
$(h_{k,\ell})_{0 \leq \ell < k \leq n+1}$ and $h$.
\end{dfn}

As in the previous example, the clustering property of $\Phi$ follows essentially by definition. 
Monotonicity, on the other hand, is inherited from that of the triangular array. To this end, we shall say that triangular array $\underline{h} = \lbr h_{k,\ell}\rbr$ is non-decreasing if
 \be{eq:h-lex-order} 
  k \leq k'\ {\rm and} \ \ell \leq \ell'\ \Rightarrow h_{k,\ell} \leq h_{k',\ell'} \,.
 \ee
 
\begin{prop}
\label{p:1.7}
Let $n \geq 1$. Suppose that $\{h_{k,\ell}\}_{0 \leq \ell < k \leq n+1}$ is a non-decreasing triangular array and that $h \geq h_{n+1 , n}$. 
Then a {\em second order branching clustering function} with parameters $\{h_{k,\ell}\}_{0 \leq \ell < k \leq n+1}$ and $h$ is a monotone clustering function.
\end{prop}

Turning to the question of wetting, as in the previous example we suppose that an infinite triangular array $(h_{k,\ell})_{0 \leq \ell < k}$ is given, and for $n \geq 1$ let $\Phi_n \equiv \Phi_n^{\{h_{k,\ell}\},h_{n+1,n}}$ be 
a {\em second order branching clustering function} with parameters $(h_{k,\ell})_{0 \leq \ell < k \leq n+1}$ and $h = h_{n+1, n}$.

\begin{prop} 
\label{prop:B2-free-energy-exist}
Suppose that $(h_{k,\ell})_{0 \leq \ell < k}$ is non-decreasing. Then, the free energy exists if
\begin{align} \label{eq:B2-free-energy-exist} 
    \sum_{k} 2^{-k} h_{k, k-1}  < \infty \,.
\end{align}
\end{prop}

For a sufficient condition on $\{h_{k,\ell}\}$ for the existence of a wetting transition
\begin{prop}
\label{p:1.13}
Suppose that $(h_{k,\ell})_{0 \leq \ell < k}$ is non-decreasing and that~\eqref{eq:B2-free-energy-exist} holds. Then, there is a wetting transition at 
$J^* \in \big(2 \ln 2 - 2e^{-1} \kappa_2,\, \infty\big)$ if
\begin{equation}
\label{e:1.36}
\kappa_2 := \sup\limits_{k} \sum\limits_{\ell=0}^{k-1} 2^\ell e^{-h_{k, \ell}} < \infty.
\end{equation}
\end{prop}

Turning to the necessary condition, if $v = (v_{n_1,n_2})_{n_1, n_2 \geq 0}$ is a doubly indexed sequence, then we denote its (bi-variate) Laplace Transform by $\hat{v}(s_1,s_2)$. That is,
    \begin{align}
        \hat{v} \lb s_1, s_2 \rb = \sum_{n_1=0}^\infty \sum_{n_2=0}^\infty e^{- s_1 n_1 - s_2 n_2}\, v_{n_1, n_2}\,,
    \end{align}
whenever that the series converges absolutely.
\begin{prop} \label{prop:B2-general-necessary}
Suppose that $(h_{k,\ell})_{0 \leq \ell < k}$ is non-decreasing and that~\eqref{eq:B2-free-energy-exist} holds. Then, there is no wetting transition if
\begin{align} \label{eq:B2-general-assumption-for-no-wetting}
    \lim_{s \downarrow 0} \big( \ln{ \frac{1}{s}} - 2 s^2 \wh{g^+} \lb s, 2s \rb\big) = \infty \,,
\end{align}
where $\{g_{\ell,d}\}$ is a doubly indexed sequence defined via,
    \begin{align} \label{eq:b2-g-def}
        g_{\ell, d} := 
        \left\{
        \begin{array}{lll}
        h_{\ell + d, \ell} - (\ln{2}) \ell & \quad & l \geq 0, d \geq 1 \,,\\
		0 & & l \geq 0, d = 0 \,.
        \end{array}
        \right.
    \end{align}
\end{prop}

By making stronger assumptions on the structure of $\{h_{k,\ell}\}$, we have the following simpler necessary condition.
\begin{corollary} \label{cor: b2 additive necessary}
Suppose that $(h_{k,\ell})_{0 \leq \ell < k}$ is non-decreasing, that~\eqref{eq:B2-free-energy-exist} holds and that 
\begin{equation}
\label{eq:B2-necessary-additive-structure1}
	h_{\ell+d, \ell} = h^{(1)}_\ell + h^{(2)}_d
	\; ; \quad \ell,d \geq 0 \,,
\end{equation}    
for two non-negative sequences $(h^{(1)}_k)_{k \geq 0}$ and $(h^{(2)}_k)_{k \ge 0}$. Then there is no wetting transition if
\begin{align} \label{eq:B2-additive-necessary}
        \lim_{k \to \infty}\, (\ln 2)k + \ln{k} - h_{k}^{\lb 1 \rb} - h_{\lfloor k / 2 \rfloor}^{\lb 2\rb} = \infty \,.
\end{align}
\end{corollary}
\begin{rem}
When $h_{k,\ell} \equiv h_\ell$ does not depend on $k$, the function $\Phi_n$ becomes a first order branching clustering functions with $(h_\ell)_{\ell \geq 0}$ and $h = h_n$. In this case we may take $h^{(2)}_d \equiv 0$ and 
$h_\ell^{(1)} = h_\ell$, in decomposition~\eqref{eq:B2-necessary-additive-structure1}. Then it is not difficult to see that conditions~\eqref{eq:B2-general-assumption-for-no-wetting} and~\eqref{eq:B2-additive-necessary} identify with conditions~\eqref{e:sufficient lim b1} and~\eqref{eq:B1-wetting-necessary-general-main} respectively. We also note that, the assumption on the positivity of the two sequences in the decomposition can be weakened to requiring that $(h^{(1)}_k)^-, (h^{(2)}_k)^- = o(k \log k)$, with minor changes to the proof.
\end{rem}
 
\subsubsection{Clustering via capacity}
Lastly we show that the (relative) capacity of a set could also be used to define a clustering function and consequently a PWC measure. While this is again a stand-alone model, it is non unnatural, as capacity of a set is inversely related to its degree of clustering. Formally, let $\calC = \lbr \sfC_e\rbr$ be a family of conductances, such that $\sfC_e = \sfC_{\abs{e}}$, 
 where $\abs{e} := \abs{u}$ if $e = e (u )$ is the edge adjacent to $u$ in the direction of the root $0$. Recall that the capacity of a set $A$ is
\begin{equation}
\label{eq:Cap} 
 {\rm CAP}_\calC (A ) := \min_{f (\mathsf{0})=1,\ f_{|_A} = 0 } \sum_{u\neq \mathsf{0}} 
 \sfC_{e (u )}\lb \nabla_{e (u )} f\rb^2 .
\end{equation}

A motivation for using capacity as a clustering function can be seen by recalling the representation of the former via escape probabilities of a continuous time random walk on $\bbT^{(n)}$, with $\calC_e$ as its transition rate across the un-oriented bond $e$. Letting $S_t$ be the associated discrete time jump chain, 
$\bbP^{\rm RW}_{u}$ be the distribution 
of  $S$ on $\bbT^{(n)}$ starting at $u$, and $\tau_A$ be the first 
return time to $A \subseteq \bbT^{(n)}$, we have
\be{eq:CAP-A-1} 
{\rm CAP}_\calC (A ) = \bbP_{\mathsf{0}}^{\rm RW}\lb \tau_A <\tau_{\mathsf{0}}\rb\cdot 
\sum_{z\sim\mathsf{0}}\sfC_{e (z )} = 
\sfC_0 \sum_{v\in A} \bbP_{v}^{\rm RW}\lb 
\tau_{\mathsf{0}} <\tau_A\rb
\ee

In particular, if $A$ and $B$ are disjoint, then 
\begin{equation*} 
\begin{split} 	
\bbP_{\mathsf{0}}^{\rm RW}\lb \tau_{A\cup B} <\tau_{\mathsf{0}}\rb 
 &= 
 \bbP_{\mathsf{0}}^{\rm RW}\lb \tau_A <\tau_{\mathsf{0}}\wedge \tau_B\rb + 
 \bbP_{\mathsf{0}}^{\rm RW}\lb \tau_B <\tau_{\mathsf{0}}\wedge \tau_A\rb \\
 &\leq \bbP_{\mathsf{0}}^{\rm RW}\lb \tau_{A} <\tau_{\mathsf{0}}\rb 
 + 
\bbP_{\mathsf{0}}^{\rm RW}\lb \tau_{B} <\tau_{\mathsf{0}}\rb , 
\end{split}
\end{equation*}
which implies \eqref{eq:MC-b}. 

It takes a bit more work to show,
\begin{prop}	
\label{p:1.5}
Let $\Phi(A) := {\rm CAP}_{\cC}(A)$. Then $\Phi$ is a monotone clustering function.
\end{prop}

Turning to wetting. As it turns out, there is no wetting transition in this case.
\begin{prop} 
\label{lem:Cap-NotWet}
There exists $C_0 < \infty$ such that for all $n$ and $A \subseteq \bbL_0^{(n)}$, 
\be{eq:Cap-Bound} 
{\rm CAP}_\calC (A ) \leq C_0 \abs{A} .
\ee
Consequently, there is no wetting transition in this case.
\end{prop} 

\subsection{Discussion and open questions}
\label{s:1.4}
Let us make some concluding remarks. First, aside from intrinsic interest, the motivation behind studying PWC measures driven by first and second order branching clustering functions is that they can be used to study the wetting transition in the more physical cases of pinned random surfaces. Indeed, in a sequel paper~\cite{sequel} we study in more detail the PWC measure of Example~\ref{s:1.1.2} in the case that the field $h$ is the restriction to the leaves of the discrete Gaussian free field $h'$ on $\bbT_n$, namely when $V(s)=s^2$ in~\eqref{e:1.26}.

The strong FKG property of the law of $h'$ which still holds under the conditioning on positivity at all leaves, will then allow us to show that
\begin{equation}
\label{e:1.61}
	\Phi^V_n(A) = \Phi^{\{h_{k,\ell}, h_{n+1,n}\}}_n(A) + O(1)
	\quad ; \qquad A \subseteq \bbL_0^{(n)} \,,
\end{equation}
where $\Phi^V_n$ is the clustering function in Example~\ref{s:1.1.2}, while
$\Phi_n^{\{h_{k,\ell}\},h_{n+1,n}}$ is the second order branching clustering function of Example~\ref{s:1.2.4} with
\begin{equation}
	h_{k,\ell} := (\ln 2) \ell + \frac{3}{2} \ln^+ (\ell \wedge (k-\ell)) 
	\; ; \quad 0 \leq \ell \leq k \leq n+1 \,.
\end{equation}
As the conditions in Proposition~\ref{p:1.13} hold, the PWC measures associated with $\Phi_n^{\{h_{k,\ell}\},h_{n+1,n}}$ exhibit a wetting transition. 

By monotonicity of the canonical free energy in $\Phi$, domination~\eqref{e:1.61} and Corollary~\ref{c:1.9}, this implies that there is a wetting transition also for the PWC measures associated with $\Phi^V_n$. This gives an alternative proof to the Caputo-Velenik argument for the existence of a wetting transition for the DGFF on the tree subject to a hard wall (cf.~\cite{CV2000}), which is more constructive in nature. We note that one can have a similar comparison of $\Phi_n^V$ with a first order branching cluster functions w.r.t. the sequence $h_k = (\ln 2) k + \ln k$. The trouble is that this case falls exactly ``between'' the necessary and sufficient conditions and Example~\ref{sss:2.3} cannot be used to rule out or in the existence of a wetting transition.

Many questions concerning the PWC model and Examples~\ref{sss:2.3} and~\ref{s:1.2.4} in particular remain. Let us now mention some of them. When there is a wetting transition at $J^* \in (-\infty, \infty)$, a natural question is whether $\rho$ is continuous at $J^*$ or not, i.e. what the order of the wetting transition is. This question is directly related to the curvature of $\omega$ at $0$, as can be seen by~\eqref{eq:Phi-psif}. Other questions concern the stochastic geometry of pinned sites in the complete and partial wetting regime. These can be phrased in term of the fractal nature of the dry-set (e.g. typical number of dry points in a ball of radius $\alpha n$ in graph-distance, or correlations between the events that two sites are pinned. 

Both answers in the generality of the basic theory and for Examples~\ref{sss:2.3} \linebreak and~\ref{s:1.2.4} are desirable. The latter, because relation~\eqref{e:1.61} will likely allow us to deduce conclusions about the dry-set in the pinned random surfaces case, where many of the above questions remain open.

Lastly, while we treat the binary tree here, a more interesting underlying graph would naturally be a $d$-dimensional lattice. While these two graphs seem very different, a tree-like structure can be extracted from a $d$-dimensional box, by, e.g., repeatedly partitioning the box to $2^d$ sub-boxes, each having half the side-length of the original one. 
As in the tree case, this may provide additional tools for studying the wetting transition in the case of $d$-dimensional random surfaces, and in particular the $d=2$ case, where the current theory is quite lacking. We believe that the general theory, and Examples~\ref{sss:2.3} and~\ref{s:1.2.4} could be adapted to the lattice case, as well as the comparison between the latter and the DGFF when $d=2$, but we leave this to a future work.

\subsection*{Paper outline}
The remainder of the paper is organized as follows. In Section~\ref{s:1} we provide proofs for the general theory of wetting as stated in the introduction. Sections~\ref{s:2} and~\ref{s:3} include the proofs for the statements concerning PWC measures driven by first and second order branching clustering functions, respectively. Section~\ref{s:4} treats the case of clustering function defined in terms of capacity. 

\section{General theory of wetting}
\label{s:1}
In this section we provide proofs for the statements of the general theory developed in the Subsection~\ref{s:1.1}

\begin{proof}[Proof of Proposition~\ref{prop:Phi-FE-A}]
	Fix dyadic $\epsilon_r, \epsilon_\ell \in (0,1)$
	and set  $\epsilon = \frac{1}{2}\lb \epsilon_r +\epsilon_\ell\rb$.
Clearly both $\epsilon_r 2^{n-1}$ and $\epsilon_\ell
2^{n-1}$ are in $\bbN$ for all $n$ sufficiently large. 	 
 In the latter case  
 if $A_r$ is a subset of $\bbL_{0, \sfr}^{(n)}$ of cardinality $\epsilon_r 2^{n-1}$  and 
$A_\ell$ is a subset of $\bbL_{0, \sfl}^{(n)}$ of cardinality $\epsilon_\ell 2^{n-1}$, then $A = A_r\cup A_\ell$ 
has cardinality
$\epsilon 2^n$. Therefore, by monotonicity and in view of assumption~\eqref{eq:Phi-A},   
\be{eq:Phi-GrGl}
\bbW_n (\epsilon) \geq {\rm e}^{-2\gamma_n} 
\bbW_{n-1}(\epsilon_r)\bbW_{n-1} (\epsilon_\ell). 
\ee

Take $\epsilon = \epsilon_\ell = \epsilon_r$. 
Then standard sub-additivity arguments imply that $\omega (\epsilon )$ is well defined and satisfies 
\be{eq:Phi-lb-frf}
\ent{\epsilon}  
\geq \omega (\epsilon )\geq \omega_n (\epsilon ) - 2\sum_{n+1}^\infty \gamma_k 2^{-k}
> -\infty . 
\ee
where $\ent{\epsilon}$ is as in~\eqref{eq:Phi-ent}. Going back to \eqref{eq:Phi-GrGl} we infer that under assumption~\eqref{eq:Phi-A}, 
\[ 
\omega\lb \frac{\epsilon_{\sfr}+\epsilon_{\sfl}}{2}\rb \geq \frac{1}{2}\lb \omega (\epsilon_{\sfr} ) + 
 \omega (\epsilon_{\sfl} )\rb, 
\]
for general dyadic $\epsilon_{\sfr/\sfl}\in [0,1]$. By iterations, this  is concavity.
\end{proof}

Next, we prove
\begin{proof}[Proof of Theorem~\ref{thm:Phi-wet}]
Convexity and hence continuity of the right hand side \linebreak of~\eqref{eq:Phi-psif} follows by concavity of $\omega$. At the same time, abbreviating $\omega_n \equiv \omega_n^{\Phi_n}$, we have 
\[ 
\zeta_n(J) =  \frac{1}{2^n} \ln 
\sum_{\eps \in [0,1]} 
{\rm e}^{2^n \left[ J\eps + \omega_n (\eps ) \right]} 
\]
Bounding the sum using the maximal summand we have
\begin{equation}
	\label{e:2.4}
\max_{\eps \in [0,1]} \lbr J\epsilon + \omega_n 
(\eps )\rbr 
\leq \zeta_n(J)
 \leq  \max_{\eps \in [0,1]} \lbr 
J\epsilon + \omega_n (\eps )\rbr 
+ \frac{n\ln 2}{2^n },
\end{equation}
with $\epsilon$ in the summation and maximum restricted to $2^{-n}\bbN_0$. 
Approximating an argmax of~\eqref{eq:Phi-psif} via a dyadic $\epsilon$ and using the convergence of $\omega_n$ to $\omega$ shows that the limit inferior of the lower bound in~\eqref{e:2.4} is at least the right hand side~\eqref{eq:Phi-psif}. On the other hand thanks to the second inequality in~\eqref{eq:Phi-lb-frf} we may replace $\omega_n$ by $\omega$ in the upper bound in~\eqref{e:2.4} at the cost of an error that tends to $0$ with $n$. This shows that the limit superior of the upper bound also tends to the right hand side of~\eqref{eq:Phi-psif}.
\end{proof}

\begin{proof}[Proof of Corollary~\ref{c:1.9}]
Since $\omega$ is concave and $\omega(0) = 0$, if $(-J) \geq \frac{\rmd^+}{\rmd \epsilon} \omega(0)$ then for all $\epsilon \in [0,1]$
\begin{equation}
\omega(\epsilon) \leq (-J)\epsilon \,.
\end{equation}
Also, by definition of the right derivative, if 
$(-J) < \frac{\rmd^+}{\rmd \epsilon} \omega(0)$ then in a right neighborhood of $0$ we must have
\begin{equation}
\omega(\epsilon) > (-J) \epsilon
\end{equation}
\end{proof}

\begin{proof}[Proof of Corollary~\ref{c:1.10}]
Let $J$ be such that the maximum in~\eqref{eq:Phi-psif} is attained uniquely at $\epsilon_0 \in [0,1]$. Then
for all $\delta > 0$, by continuity, the same maximum restricted further to $(\epsilon_0 - \delta, \epsilon_0+\delta)^\rmc$ must be smaller than $\zeta(J)-\eta$ for some $\eta > 0$. In particular, bounding the sum using the maximal summand and appealing to~\eqref{eq:Phi-lb-frf}, we have
\begin{equation}
\begin{split}
\zeta_n(J; \delta) :&= 
\frac{1}{2^n} \ln 
\sum_{|2^{-n}|A| - \epsilon_0| \geq \delta}
{\rm e}^{J \abs{A} - \Phi (A)} 
= 
\frac{1}{2^n} \ln 
\sum_{|\eps - \eps_0| \geq \delta} 
{\rm e}^{2^n \left[ J\eps + \omega_n (\eps ) \right]} \\
& \leq  \max_{|\eps - \eps_0| \geq \delta} \lbr 
J\epsilon + \omega_n (\eps )\rbr 
+ \frac{n\ln 2}{2^n }
\leq \zeta(J) - \eta
+ 2\sum_{n+1}^\infty \gamma_k 2^{-k}
 + \frac{n\ln 2}{2^n } 
\end{split}
\end{equation}

Since
\begin{equation}
	\frac{1}{2^{n}} \ln \rmP_n^{\Phi,J} \Big( \big|2^{-n}|A| - \epsilon_0 \big| \geq \delta \Big) 
	= \zeta_n(J; \delta) - \zeta_n(J) \,,
\end{equation}
we thus get
\begin{equation}
	\limsup_{n \to \infty} \, \ln \rmP_n^{\Phi,J} \Big( \big|2^{-n}|A| - \epsilon_0 \big| \geq 
	\delta \Big) \leq -\eta \,,
\end{equation}
which implies exponential convergence in probability to $\epsilon_0$ of $2^{-n} |A|$, and thanks to the boundedness of this quantity also in $L^1$.

In the same vein, for any $\Delta > 0$,
\begin{equation}
\zeta(J) + \epsilon_0 \Delta 
\leq \zeta(J+\Delta) \leq 
\max \Big\{\zeta(J) + (\epsilon_0 + \delta) \Delta \,,\,
\zeta(J) - \eta + \Delta \Big\} \,,
\end{equation}
Subtracting $\zeta(J)$ and then dividing by $\Delta$ and finally taking $\Delta \downarrow 0$ follows by $\delta \to 0$, shows that the right derivative of $\zeta$ exists at $J$ and equal to $\epsilon_0$. A symmetric argument shows the same for the left derivative.

Lastly, since $J \epsilon + \omega(\epsilon)$ is concave, if it has more than one maximizer on $[0,1]$, it must be a constant on some sub-interval of $[0,1]$, which implies that $\omega(\epsilon)$ is linear on that domain with slope $-J$. Since these sub-intervals must be disjoint for different $J$-s, there could be at most countably many of them.
\end{proof}

\section{Clustering via first order branching patterns}
\label{s:2}
In this section we provide proofs for the various statements in Subsection~\ref{sss:2.3}.

\subsection{Monotoncity}
\begin{proof}[Proof of Proposition~\ref{p:1.6}] 
For $k=0,\dots, n$, let $\beta_k = \sum_0^{k} b_j$ be the number of branching points in all the generations starting from generation $k$ (notice that we refer to the leaves as branching points). 
Alternatively, $b_k = \beta_k - \beta_{k-1}$. 
Of course, 
if $\abs{A} = N+1$, then $\beta_n (A) = \abs{\calB (A)}= N + N + 1 = 2N + 1$. Summing by parts, 
\be{eq:BP-repr}
 \Phi_n (A) = \sum_0^n h_k b_{k} + h = Nh_n - \sum_{k=0}^{n-1} (h_{k+1} -h_k )\beta_{k}(A) + h .
\ee
First we will show \eqref{eq:MC-a}. Let $A\prec C$ and let $\sigma: A \rightarrow C$ be the bijection from the definition of $A \prec C$. 
We will prove that $\beta_k (A) \geq \beta_k (C)$ by induction on $|A|=|C|=N+1$. 
If $N=0$ then $|\calB (A)| = |\calB (C)| = 0$ so $\beta_k (A) \geq \beta_k (C)$ trivially for all $k=0,\dots, n$.
Assume that $\beta_k (A) \geq \beta_k (C)$ for all $A \prec C$ of cardinality $N+1$. Let $A \prec C$, $|A|=|C|=N+2$. Label $A=\{1,\dots,N+2\}$, $C=\{\sigma_1, \dots, \sigma_{N+2}\}$. Define $m_i = \min_{i \neq j} |i \wedge j|$. Without loss of generality, assume $|1 \wedge 2| = m_1 = \min_i m_i$. Consider now $\hat A = A \setminus \lbr1\rbr$, $\hat C = C \setminus \lbr\sigma_1\rbr$. By restricting $\sigma$ to $\hat A$ we get $\hat A \prec \hat C$. By the induction hypothesis $\beta_k (\hat A) \geq \beta_k (\hat C)$ for all $k=0,\dots,n$. Since adding one leaf to $\hat A$ implies a single additional branching point at generation $m_i$ we get
\[ \beta_k (A) = 
\begin{cases}
\beta_k (\hat A) +1 & k \geq m_i \\
\beta_k (\hat A) & k < m_i
\end{cases}
\]
Denote $m^{C}_1=\min_{j \neq 1} |\sigma_1 \wedge \sigma_j|$. Then 
\[ \beta_k (C) = 
\begin{cases}
\beta_k (\hat C) +1 & k \geq m^{C}_i \\
\beta_k (\hat C) & k < m^{C}_i
\end{cases}
\]
Since $m_1 \leq m^{C}_1$,
\[ \forall k=0,\dots,n \quad \beta_k (A) \geq \beta_k (C) .\]
Since $h_k$ is non-decreasing, this implies \eqref{eq:MC-a}.
Moving on to \eqref{eq:MC-b}, let $A,C \subseteq \bbL_0$ be disjoint, $A\cap C = \emptyset$. Then
\[|A\cup C|=|A|+|C| \Longrightarrow \beta_n (A\cup C) = \beta_n (A) +\beta_n (C) + 1.\]
First, consider the case where $\calB (A) \cap \calB (C) = \emptyset$. Because $\calB (A), \calB (C) \subseteq \calB (A\cup C)$ we get $\calB (A \cup C) = \calB (A) \cup \calB (C) \cup \lbr z \rbr$ so:
\[b_{|z|} (A\cup C) = b_{|z|} (A) + b_{|z|} (C) + 1\]
\[\forall k\in \lbr 0,\dots,n \rbr, k\neq |z| \quad b_{k} (A\cup C) = b_{k} (A) + b_{k} (C)\]
Therefore
\[\forall k=0,\dots,n \quad \beta_k (A \cup C) \geq \beta_k (A) + \beta_k (C).\]
Now suppose that $\calB (A) \cap \calB (C) \neq \emptyset$. Define
\[\calB_{AC} := \calB (A) \cap \calB (C), \ |\calB_{AC}|=N>0\]
\[\calB_{A \Delta C} := \calB (A \cup C) \setminus (\calB (A) \cup \calB (C))\]
since $\calB (A), \calB (C) \subseteq \calB (A\cup C)$,
\begin{align*}
    |\calB (A) \cup \calB (C)| &= |\calB (A)| + |\calB (C)| - |\calB (A) \cap \calB (C)| \\
    &= \beta_n (A) + \beta_n (C) - N
\end{align*}
and
\begin{align*}
    |\calB_{A \Delta C}| &= |\calB (A\cup C)| - |\calB (A) \cup \calB (C)| \\
    &= \beta_n (A \cup C) - (\beta_n (A) + \beta_n (C) - N) \\
    &= N+1 \,.
\end{align*}
Every $z\in \calB_{AC}$ is counted twice in $\beta_k (A) + \beta_k (C)$, for $k \geq |z|$, but only once in $\beta_k (A \cup C)$.
On the other hand, $z\in \calB_{A \Delta C}$ are only counted in $\beta_k (A \cup C)$, $k\geq |z|$. Hence, we get \eqref{eq:MC-b} by constructing a bijection $\tau:  \calB_{A \Delta C} \rightarrow \calB_{AC} \cup \lbr t \rbr$ satisfying:
\[\forall z\in \calB_{A \Delta C} \quad |\tau (z)| \geq |z| \text{ or } \tau (z) = t\]
where $t$ is a placeholder node.

Label $\calB_{A \Delta C}=\lbr 1,\dots, N+1 \rbr$.

\begin{lem}
\label{lem:AC-DAC-a}
Any $x\in \calB_{AC}$ has no ancestor in $\calB_{A \Delta C}$.
\end{lem}

\begin{proof}
Assume by way of contradiction that $\exists x\in \calB_{AC},z\in \calB_{A \Delta C}$ such that $z$ is an ancestor of $x$. 
By the definitions of $\calB_{AC}$ and $\calB_{A \Delta C}$ there exist $u_1, u_2, u_3\in A$, $v_1, v_2, v_3\in C$ such that $x=u_1\wedge u_2=v_1 \wedge v_2$ and $z=u_3 \wedge v_3$. We call a path $P=\langle s_1,\dots,s_k \rangle$ in $\bbT^{(n)}$ an ascending path if $|s_{i+1}|>|s_i|$ for all $i=1,\dots,k-1$. 
Since $z$ is an ancestor of $x$, there exists an ascending path $P_1=\langle x,\dots,z \rangle$. 
Denote by $z_1,z_2$ the direct children of $z$. 
Without loss of generality, assume $u_3$ is a descendant of $z_1$, $v_3$ is a descendant of $z_2$, and that $P_1=\langle x,\dots,z_1, z \rangle$. 
Let $P_2=\langle v_1,\dots,x \rangle$ be the ascending path from $v_1$ to $x$, and $P_3=\langle v_3,\dots,z_2, z \rangle$ be the ascending path from $v_3$ to $z$. 
Since $z_1\neq z_2$ and this is a tree, $\lbr P_2 \cdot P_1 \rbr \cap \lbr P_3 \rbr = \lbr z \rbr$, i.e. the simple paths from $v_1$ and $v_3$ to the root intersect for the first time at $z$ so $z=v_1 \wedge v_3$. This means that $z\in \calB (C)$ and therefore $z\notin \calB_{A \Delta C}$ which is a contradiction.
\end{proof}

\begin{lem}
\label{lem:AC-DAC-b}
\[\forall z_1 \neq z_2\in \calB_{A \Delta C}, \quad x=z_1\wedge z_2\in \calB_{AC}.\]
\end{lem}

\begin{proof}
$z_1=u_1\wedge v_1$, $z_2=u_2\wedge v_2$ where $u_1, u_2\in A$ and $v_1, v_2\in C$. 
Then $x$ is an ancestor of $v_1$ and $v_2$ and they are in different branches starting at $x$ (symmetrically, also for $u_1$ and $u_2$). 
Hence $x\in \calB_{AC}$.
\end{proof}

\begin{lem}
\label{lem:AC-DAC-c}
\[\forall x_1 \neq x_2\in \calB_{AC}, \quad x=x_1\wedge x_2\in \calB_{AC}.\]
\end{lem}

\begin{proof}
By definition of $\calB_{AC}$, $x_1=u_1 \wedge u_2=v_1 \wedge v_2$ and $x_2=u_1' \wedge u_2'=v_1' \wedge v_2'$ where $u_1,u_2, u_1',u_2'\in A$ and $v_1,v_2, v_1',v_2'\in C$. If $x_1=x_2$ then the lemma is trivial. Otherwise, $x_1 \neq x_2$. Since $x_1 \neq x_2$ and $x$ is the lowest (closest to the leaves) common ancestor of $x_1$ and $x_2$, $\lbr u_1,u_2,v_1,v_2 \rbr$ and $\lbr u_1',u_2',v_1',v_2' \rbr$ are in different branches of $x$ so $x=u_1\wedge u_1'=v_1\wedge v_1'\in \calB_{AC}$.
\end{proof}

\begin{lem}
\label{lem:AC-DAC-d}
\[\forall x\in \calB_{AC}, \exists z_1 \neq z_2\in (\calB_{A \Delta C} \cup \calB_{AC}) \setminus \lbr x \rbr : x=z_1 \wedge z_2\]
\end{lem}

\begin{proof}
Let $x\in \calB_{AC}$. $x=u_1\wedge u_2=v_1\wedge v_2$ for $u_1, u_2\in A$ and $v_1, v_2\in C$. 
Denote $x_1, x_2$ the direct children of $x$. 
Without loss of generality, $x_1$ is an ancestor of $u_1, v_1$ and $x_2$ is an ancestor of $u_2, v_2$. 
Then $|u_1\wedge v_1| \leq |x_1| < |x|$, $|u_2 \wedge v_2| \leq |x_2| < |x|$.
For $i=1, 2$, denote $z_i^{(0)} = u_i \wedge v_i$.
If $z_i^{(0)} \notin (\calB_{A \Delta C} \cup \calB_{AC}) \setminus \lbr x \rbr$ then either $z_i^{(0)} \in \calB (A) \setminus \calB (C)$ or $z_i^{(0)} \in \calB (C) \setminus \calB (A)$. Without loss of generality assume that $z_i^{(0)} \in \calB (A) \setminus \calB (C)$.
Then there exists $u_i^{(1)} \in A$ such that $z_i^{(1)} = u_i^{(1)} \wedge v_i$ is a descendant of $z_i$ satisfying $|z_i^{(1)}| < |z_i^{(0)}|$. 
If $z_i^{(1)} \notin (\calB_{A \Delta C} \cup \calB_{AC}) \setminus \lbr x \rbr$ then either $z_i^{(1)} \in \calB (A) \setminus \calB (C)$ or $z_i^{(1)} \in \calB (C) \setminus \calB (A)$. Without loss of generality assume that $z_i^{(1)} \in \calB (A) \setminus \calB (C)$.
Then there exists $u_i^{(2)} \in A$ such that $z_i^{(2)} = u_i^{(2)} \wedge v_i$ is a descendant of $z_i^{(1)}$ satisfying $|z_i^{(2)}| < |z_i^{(1)}|$.
We can continue this process, choosing $z_i^{(0)}, \dots, z_i^{(k)}$ such that $|z_i^{(j + 1)}| < |z_i^{(j)}|$, until a node satisfying $z_i^{(k)} \in (\calB_{A \Delta C} \cup \calB_{AC}) \setminus \lbr x \rbr$ is found.
Finally, such a node can be found since if $|z_i^{(k)}|=1$ then $z_i^{(k)}$ only has two leaf descendants, one of which is in $A$ while the other is in $C$, and therefore $z_i^{(k)} \in \calB_{A \Delta C}$.
Overall, there exist $z_1, z_2 \in (\calB_{A \Delta C} \cup \calB_{AC}) \setminus \lbr x \rbr$ which are descendants of $x_1$ and $x_2$ respectively such that $x = z_1 \wedge z_2$.
\end{proof}

From the last few lemmas we deduce that the subtree induced by $\calB_{AC} \cup \calB_{A \Delta C}$, denoted $\calT_{AC}^{(0)}$, is a binary tree in which $\calB_{A \Delta C}$ are the leaves and $\calB_{AC}$ are inner vertices (there might be other non-branching inner vertices) such that all vertices in $\calB_{AC}$, and only them, are branching points of $\calT_{AC}^{(0)}$.  
This observation allows us to construct a bijection $\tau$ as described above. 
Let $\langle x_1,\dots,x_N \rangle$ be an ordering of $\calB_{AC}$ satisfying $|x_i| \leq |x_{i+1}|$ for all $i=1,\dots,N-1$. Let's start with $\tau (x_1)$. 
By the minimality of $|x_1|$, it has two leaf descendants - $z_{1}^{(1)},z_{2}^{(1)}\in \calB_{A \Delta C}$. Set $\tau (z_{1}^{(1)}) = x_1$ and $\calT_{AC}^{(1)}$ to be the tree induced by $\calB_{AC}^{(1)} \cup \calB_{A \Delta C}^{(1)}$ where we set $\calB_{AC}^{(1)}=\calB_{AC} \setminus \lbr x_1 \rbr$ and  $\calB_{A \Delta C}^{(1)} = \calB_{A \Delta C} \setminus \lbr z_{1}^{(1)} \rbr$. 
This means that the vertices of $\calT_{AC}^{(1)}$ no longer match the original branching pattern, but this does not affect our construction of $\tau$. 
Notice that the new tree still satisfies all the previous statements with the new $\calB_{AC}^{(1)}$ and $\calB_{A \Delta C}^{(1)}$. 
We can continue this process, suppose we set $\tau (z_{1}^{(1)}) = x_1,\dots, \tau (z_{1}^{(k)}) = x_k$, and the corresponding $\calT_{AC}^{(k)}$, $\calB_{AC}^{(k)}$ and $\calB_{A \Delta C}^{(k)}$. 
As before, by the minimality of $|x_{k+1}|$ in $\lbr x_{k+1},\dots,x_N\rbr $ it has two leaf descendants in $\calT_{AC}^{(k)}$; $z_{1}^{(k+1)}, z_{2}^{(k+1)}\in \calB_{A \Delta C}^{(k)}$. 
Define $\tau (z_{1}^{(k+1)}) = x_{k+1}$ and $\calT_{AC}^{(k+1)}$, $\calB_{AC}^{(k+1)} = \calB_{AC}^{(k)} \setminus \lbr x_{k+1} \rbr$ and $\calB_{A \Delta C}^{(k+1)} = \calB_{A \Delta C}^{(k)} \setminus \lbr z_1^{(k+1)} \rbr$. 
By construction, for all $i\in \lbr 1,\dots,N \rbr$, $z_{1}^{(i)}$ is a descendant of $\tau (z_{1}^{(i)})$ in $\calT_{AC}$ so,
\be{eq:B1-tau-property}
\forall i\in \lbr 1,\dots,N \rbr, \quad  |\tau (z_{1}^{(i)})| > |z_{1}^{(i)}|
\ee
Since $|\calB_{A \Delta C}| = |\calB_{AC}| +  1$, after assigning $t$ to the remaining $z \in \calB_{A \Delta C}$ we get the required bijection and \eqref{eq:MC-b} is implied.
\end{proof}

\subsection{Existence of free energy}

We begin by proving the existence and finiteness of the free energy.
\begin{proof}[Proof of Proposition~\ref{p:1.8}]
The ``if'' part follows from the general statement of \linebreak Proposition~\ref{prop:Phi-FE-A} by choosing $\gamma_n = h_n$ which is summable if \eqref{eq:Lev1Exist} holds.
For the other direction, by monotonicity of $h$ it is enough to consider $\epsilon$ of the form $\epsilon = 2^{-m_0}$; $m_0\in\bbN_0$.
Furthermore, since $h_n$ is assumed to be non-decreasing, we can minimize $\Phi_n$ by performing all branchings as close to the leaves as possible.
That is, by choosing $\underline{b}_n = (2^{n - m_0 - 1}, \dots, 4,2,1,0, \dots, 0)$. 
The next lower bound of $\Phi (A)$ follows from this reasoning:
 \[ 
 \Phi_n (A ) \geq 
  \sum_{\ell=0}^n h_\ell b_\ell (A ) \geq \sum_{\ell= 0}^{n-m_0} 2^{n-m_0 - \ell} h_\ell 
  = 2^{n-m_0} \sum_{\ell= 0}^{n-m_0} 2^{- \ell} h_\ell = a_0 \sum_{\ell= 0}^{n-m_0} 2^{- \ell} h_\ell
 \] 
 for every $A\subseteq \bbL^{(n)}$ of cardinality $a_0 = 2^{n-m_0}$. The case $a_0 = 0$ is trivial. Otherwise, since the total 
 number of such $A$-s is 
 \be{eq:WT-ent-mnot}
 \binom{2^n}{2^{-m_0} 2^n} 
 \sim   {\rm e}^{a_0 m_0 \ln 2} , 
 \ee
it follows that $\omega (\epsilon_0 ) = -\infty$ for any $m_0\in \bbN_0$, 
whenever the series in \eqref{eq:Lev1Exist} diverge. 
\end{proof}

\subsection{Sufficient condition for a wetting transition} 
Henceforth we are going to assume that condition~\eqref{eq:Lev1Exist} holds. In particular, this condition implies that $2^{-n} h_n \to 0$ as $n \to \infty$. Therefore the canonical partition function $\omega$ will not be changed if we take $h = 0$ instead of $h=h_n$ in the definition of $\Phi_n$ as a first ordering clustering function. While $\Phi_n$ is no longer monotone, we shall not make use of this property in the computations below. On the other hand, this redefinition will make these computations slightly simpler.

For what follows, if $\ul b = (b_k)_{k=0}^n$ is branching pattern of the subtree $\cT(A)$ induced by $A \subseteq \bbL_0^{(n)}$ in
$\bbT^{(n)}$, then we shall write $A \sim \ul b$. 
Given such sub-tree we shall denote by $a_\ell$ the total population at age $\ell$ or older, so that for $l \in [0,n]$, 
\be{eq:B1-ak-bk-relationship}
a_\ell = 1 +\sum_{k=\ell +1}^n b_k \,,
\; \ell = 0, \dots, n \,.
\ee
Notice that $a_0 = b_0 = |A|$.
Given $n \geq 1$ and $a_0$, we say that a sequence $\ub = \lb b_0 , 
\dots , b_n\rb$ is an {\em admissible branching pattern} for $a_0$ and write 
$\ub\sim a_0$ if there exists $A \subseteq \bbL_0^{(n)}$ such that $|A| = a_0$ and $\ul b$ is the branching pattern corresponding to $A$. It is an admissible branching pattern if it is admissible for some $a_0$. It is not difficult to check that $\ul b = (b_k)_{k=0}^n$ is admissible if and only if $b_k \in \bbN$ for all $k \in [0,n]$,
\be{eq:L1-comp} 
 a_0 = 1 +\sum_{j=1}^{n} b_j \,,
 \qquad \text{and} \qquad 
 b_k \leq a_k = 1+\sum_{j=k+1}^n b_j 
 \,,\,\, \forall k=1, \dots, n  \,.
\ee
We can therefore write the canonical partition function as
\be{eq:L1-Gpf} 
\bbW_n (a_0 ) := 
\sum_{\ub\sim a_0} \ent{\underline b} {\rm e}^{-\Phi_n (\ub)}\,,
\ee
where $\Phi_n(\ul b) = {\rm e}^{-\sum_{\ell = 0}^n h_\ell b_\ell }$ and 
\begin{equation}
\ent{\underline b} := \# \{A \subseteq \bbL^{(n)}_0 :\: A \sim \underline b \} \,.
\end{equation}

\begin{lem} \label{lem: entropy of first order branching pattern}
    The entropy of a branching pattern $\ul{b}$ is given by
    \begin{align} \label{eq: entropy of first order branching pattern}
        \ent{ \ul{b} } = \prod_{k=1}^n \binom{a_k}{b_k} 2^{a_k -b_k}\,.
    \end{align}
\end{lem}
\begin{proof}
    As previously stated, a first order branching pattern $\ul{b}$ completely determines the number of vertices in each layer of the induced sub-tree $a_k$.
    At each layer $k$, there are $\binom{a_k}{b_k}$ ways to choose the branching points of the level.
    Then, there are $2^{a_k - b_k}$ ways to choose the descendent of each non-branching vertex.
    These decisions can be performed independently of other layers, only adhering to the required $b_k$ branching points out of $a_k$, so the total number of subsets $A \subseteq \bbL_0$ such that $A \sim \ul{b}$ is the product of all terms as in \eqref{eq: entropy of first order branching pattern}.
\end{proof}
\begin{proof}[Proof of Proposition~\ref{p:1.9}]
For $a_0 \in [0,2^n]$ we may write
\be{eq:Gn-upper} 
\begin{split} 
 \bbW_n (a_0 )  = \sum_{\ub \sim a_0} \ent{\underline b} {\rm e}^{-\Phi_n (\ub )} & = 
 \sum_{\ub\sim a_0 } \prod_{k=1}^n \binom{a_k}{b_k} 2^{a_k -b_k} {\rm e}^{-\sum_{\ell = 0}^n h_\ell b_\ell }\\
 & = 2^{n+2} e^{-a_0(2\ln 2 + h_0)} \sum_{\ub \sim a_0} \prod_{k=1}^n \binom{a_k}{b_k} {\rm e}^{-\sum_{\ell = 1}^n g_\ell b_\ell } \,,
 \end{split}
\end{equation}
where $g$ is as in~\eqref{e:gk b1} and we used summation by parts:
\begin{equation}
\label{e:3.10}
 \sum_{k=1}^n (a_k - b_k ) = (n+2) - 2a_0 + \sum_{k=1}^n k b_k \,.
\end{equation}
Using that
\begin{equation}
\prod_1^{n} \binom{a_k}{b_k} \leq \prod_1^{n} \binom{a_{k-1}}{b_k} = \frac{a_0 !}{b_n!\dots b_1!} \,,
\end{equation}
the right hand side of~\eqref{eq:Gn-upper} is further upper bounded by 
\begin{equation}
2^{ n+2} e^{- a_0(2\ln 2 + h_0)} \Big(\sum_{k=1}^n {\rm e}^{-g_k}\Big)^{a_0}
= 2^{n+2} \exp \Big(\!-a_0\Big( 2 \ln 2 + h_0  - \ln \sum_{k=1}^n {\rm e}^{-g_k }\Big) \Big) \,.
\end{equation}
Plugging in $a_0 = \epsilon 2^n$ and using Theorem~\ref{prop:B2-free-energy-exist}, we get
\begin{equation}\label{eq:Jstar-upper} 
J^* = 
    - \lim_{\epsilon \to 0} \lim_{n\to\infty}\frac{\ln \bbW_n (\epsilon 2^n)}{ \epsilon 2^n} 
    \geq 
 2 \ln 2 + h_0 - \ln \sum_{k=1}^n e^{-g_k} 
 = 2 \ln 2 + h_0 - \ln \kappa_1 \,,
\end{equation}
with $\kappa_1$ as in the statement of the theorem.
\end{proof}
\subsection{Necessary condition for a wetting transition.}
The necessary condition will follow directly from Lemma~\ref{l:3.5} below. Recall that the sequence $g$ is related to $h$ as in~\eqref{e:gk b1}.

\begin{lem}
\label{l:3.5}
Let $(h_k)_{k=0}^\infty$ be a non-decreasing sequence and that~\eqref{eq:Lev1Exist} holds. Suppose that $g_k \geq 0$ for all $k \geq 0$. Then for all small enough $s > 0$ such that $e^s - 1$ is dyadic,
\begin{equation}
\label{e:3.14}
	-J^* \geq \ln \frac{1}{s} - \big(s+O(s^2)\big) \hat{g}(s) - C \,. 
\end{equation}
\end{lem}

\begin{proof}
If $\hat{g}(s_0) = \infty$ for some $s_0 > 0$, then $\hat{g}(s) = \infty$ for all $s \leq s_0$, so that the statement holds trivially. We therefore henceforth assume that $\hat{g}(s) < \infty$ for all $s > 0$. Let $(\epsilon_j)_{j \geq 1}$ be a sequence which tends to $0$ as $j \to \infty$. 
For each $j \geq 1$ and all $n$ large enough, let $\ul b^{(j,n)} = (b^{(j,n)}_k)_{k=0}^n$ be an admissible branching 
pattern on $\bbT^{(n)}$ for $a_0^{(j)} = \epsilon_j 2^n$. By Theorem~\ref{thm:Phi-wet} and since the sum in~\eqref{eq:Gn-upper} is lower bounded by any of its terms, we have
\begin{equation}
\label{e:3.15}
	-J^* \geq \lim_{j \to \infty} \limsup_{n \to \infty}\, A_n( \ul b^{(j,n)}) - C \,,
\end{equation}
where
\begin{equation}
\label{e:3.16}
A_n(\ul b) := a_0^{-1} \sum_{k=1}^n \big( \ln{\binom{a_k}{b_k} - g_k b_k} \big) \,.
\end{equation}
It is therefore enough to find for any $s > 0$ as in the statement of the lemma, sequences $(\epsilon_j)_{j}$
and $(b^{(j,n)})_{j,n}$ such that the limit on the right hand side of~\eqref{e:3.15} is at least the  quantity on the right hand side of~\eqref{e:3.14}.

To this end, for an admissible branching pattern $\ul b = (b_k)_{k=0}^n$, we set
\begin{equation}
\label{e:3.16a}
t_k = \frac{b_k}{a_k} \quad ; \qquad k \in [1,n] \,,
\end{equation}
so that by \eqref{eq:B1-ak-bk-relationship} 
\be{eq:B1-ak-tk}
    a_{k+1} = a_k (1 + t_{k+1})^{-1} \quad , \qquad a_k = a_0 \prod_{\ell=1}^k (1 + t_\ell) ^{-1} \,,
\end{equation}
and then by Stirling's approximation
\begin{equation}
    \ln{\binom{a_k}{b_k}} = \ln{\binom{a_k}{t_k a_k}} 
    = a_k \varphi (t_k) + O (\ln{a_k}) 
    = a_0 \varphi (t_k) \prod_{\ell=1}^k (1 + t_\ell)^{-1}  + O (n) \,,
\end{equation}
where
\be{eq:def-phi}
    \varphi (t) = -t \ln{(t)} - (1 - t) \ln{(1 - t)} \,.
\ee
Plugging this in~\eqref{e:3.16}, gives,
\begin{equation}
\label{e:3.20}
A_n(\ul{b}) \geq \sum_{k=1}^n  \big(\varphi (t_k) - g_k t_k \big) \prod_{\ell=1}^k \left( 1 + t_\ell \right)^{-1} + O(a_0^{-1} n^2) \,.
\end{equation}
Now given a dyadic $t \in (0, 1)$, for all $j \geq 1$ and $n$ large enough let
$\ul{b}^{(j,n)} = (b^{(j,n)}_k)_{k=0}^n$ be the sequence defined via~\eqref{eq:B1-ak-bk-relationship} and~\eqref{e:3.16a} for
\begin{align*}
    t_k = \begin{cases}
        t & 1 \leq k \leq j \,, \\
        1 & j < k \leq n  \,.
    \end{cases} 
\end{align*}
Then for all $n$ large enough, $\ul{b}^{(j,n)}$ is an admissible branching pattern for $a_0 = 
\eps_j 2^n$ with $\eps_j := ((1 + t)/2)^j$.
Clearly $\eps_j \to 0$ as $j \to \infty$. 
At the same time, by~\eqref{e:3.20},
\begin{equation} 
\label{eq:B1-close-to-laplace-with-additional-negligible}
\begin{split}
A_n(\ul b^{(j,n)}) & \geq 
    \sum_{k = 1}^{j} (1 + t)^{-k} \left( \varphi (t) - t g_k \right) - \left( \frac{2}{1 + t} \right)^{j} \sum_{k = j + 1}^{n} 2^{-k} g_k + o(1) \,, \\
    & \geq \varphi (t) \sum_{k = 1}^{j} (1 + t)^{-k} - t \sum_{k = 1}^{j} (1 + t)^{-k} g_k  - \sum_{k = j + 1}^{n} (1+t)^{-k} g_k + o(1) \,, 
\end{split}
\end{equation}
with $o(1) \to 0$ as $n \to \infty$ for all $j$. Finiteness of the Laplace transform of $g$ for all $s > 0$ shows that the last sum tends to $0$ when $n \to \infty$ followed by $j \to \infty$. Taking these limits, we thus get, in light of~\eqref{e:3.15}, whenever $t$ is small enough,
\begin{equation}
	-J^* \geq  \frac{\varphi (t)}{t} - t \hat{g}\big(\ln (1+t)\big) - C 
		\geq  \ln \frac{1}{t} - t \hat{g}\big(\ln (1+t)\big) - C' \,.
\end{equation}
where we have used that $\varphi (t)/t = \ln{\frac{1}{t}} + O (1)$ near $0$. 
Finally, substituting $t=e^s - 1 = s + O(s^2)$ as $s \to 0$ and noting that $-\ln (s+O(s^2)) = -\ln s + O(s)$ in this regime, this gives the desired statement.
\end{proof}

We immediately get
\begin{proof}[Proof of Proposition~\ref{p:1.10}]
Since the free energy only decreases if we replace $h_k$ by $h_k \vee (\ln 2)k$, which is also non-decreasing and satisfies~\eqref{eq:Lev1Exist}, we may prove the statement under the additional assumption that $g_k \geq 0$. 
Then Lemma~\ref{l:3.5} is in force, and~\eqref{e:sufficient lim b1} gives $|s\hat{g}(s)| = s\wh{g^+}(s) = O(\ln  (1/s))$, so that $O(s^2) \hat{g}(s) \to 0$ as $s \to 0$. 
It then follows from Lemma~\ref{l:3.5} that $-J^* = \infty$.
\end{proof}

\begin{proof}[Proof of Corollary~\ref{c:1.11}]
As in in the proof of Proposition~\ref{p:1.10}, we shall assume w.l.o.g. that $g$ is non-negative. It follows from~\cite[Corollary 1c]{widder2015laplace}, and that
the Laplace transform of the (possibly negative) sequence $(u_{n} - u_{n-1})_{n \geq 1}$ is $(1-e^{-s})\hat{u}(s)+u_0 \leq 2 s\hat{u}(s)+u_0$, that
\begin{equation}
\label{e:3.23}
\liminf_{s \downarrow 0} 2 s\hat u(s) \geq \liminf_{k \to \infty} u_k - u_0 \,,
\end{equation}
for any sequence $u = (u_n)_{n \geq 0}$ for which $\hat{u}(s) < \infty$ for all $s > 0$.  Now suppose that~\eqref{eq:B1-wetting-necessary-general-main} holds and take $u=(u_k)_{k \geq 0}$ to be sequence defined via $u_k := \ln (k \vee 1) - g_k$. Then by positivity of $g$ and~\eqref{eq:B1-wetting-necessary-general-main} we have $\hat{u}(s) < \infty$ for all $s > 0$. Noting that the Laplace transform of $\ln (k \vee 1)$ is $s^{-1} \ln s^{-1} + O(s^{-1})$ as $s \to 0$, we thus get from~\eqref{e:3.23}
\begin{equation}
\liminf_{s \downarrow 0} \big( \ln \frac1s - s\hat{g}(s)\big) \geq \frac{1}{2} \liminf_{k \to \infty} (\ln k - g_k) - C = \infty\,,
\end{equation}
so that there is no wetting transition by Proposition~\ref{p:1.10}.
\end{proof}

\section{Clustering via second order branching patterns}
\label{s:3}
\subsection{Monotoncity}
\begin{proof}[Proof of Proposition~\ref{p:1.7}]
The proof of~\eqref{eq:MC-a} will be done 
by induction on the cardinality $N+1$ of 
$A\subseteq \bbL_0$.  \eqref{eq:MC-a} is trivial for $N=1$. Assume that 
\eqref{eq:MC-a} holds for cardinality $N+1$ and let $A\prec B$ be two subsets of $\bbL_0$ of 
cardinality $N+2$. 

We shall label leaves in $A$ as $1, \dots , N+2$ and leaves in $B$ as $\sigma_1 , \dots , \sigma_{N+2}$, 
where $\sigma$ is the correspondence which realizes $A\prec B$, that is $|i \wedge j|\leq 
|\sigma_i \wedge \sigma_j|$ for any $i\neq j$. Define
\be{eq:m-and-r} 
m_i = \min_{j\neq i} |i\wedge j| , \ i^*= \min\lbr j~: ~ |i \wedge j|= m_i\rbr\  \text{ and } \ 
r_i =  \min_{j\neq i, i^*} |i \wedge j| \,.
\ee
Without loss of generality we may assume that $|1 \wedge 2| = m_1 = \min_i m_i$. Note that this necessarily 
implies that $r_1 > m_1$ (otherwise $m_2 < m_1$). 
Consider now $\hat A = A\setminus \lbr 1 \rbr$ and $\hat B = B\setminus \{ \sigma_1 \}$. 
Clearly, $\hat A\prec\hat B$, and hence by the induction hypothesis $\Phi_n (\hat A )\leq \Phi_n (\hat B )$. 
By construction, 
\be{eq:Ahat-A}
 \Phi_n (A ) = \Phi_n (\hat A ) + h_{r_1 , m_1}. 
\ee
In order to compare \eqref{eq:Ahat-A} with the corresponding contribution to $\Phi_n (B)$ we
shall consider two cases.

\noindent
\case{1} 
$r_{\sigma_1} > m_{\sigma_1}$. 
Then there exists unique $j$ such that $m_{\sigma_1} = 
|\sigma_1 \wedge \sigma_j|$. 
It follows that $m_{\sigma_1} \geq m_1$ and $r_{\sigma_1}\geq r_1$. 
We are adding a branching point in generation $m_{\sigma_1}$ with branching ancestor 
in generation $r_{\sigma_1}$. 
Hence, 
\be{eq:Bhat-B1}
\Phi_n(B ) = \Phi_n (\hat B) + h_{r_{\sigma_1} , m_{\sigma_1}} \geq \Phi_n (\hat A) + h_{r_1 , m_1} = \Phi_n  (A). 
\ee

\noindent
\case{2} $r_{\sigma_1} = m_{\sigma_1}$. In this case we add a branching point in generation 
$r_{\sigma_1} \geq r_1$. So its branching ancestor belongs to generation $r$ with $r >r_1$. 
Since $r_{\sigma_1} = m_{\sigma_1}$ it also has a branching descendant in generation $m$ with 
$m\geq m_1$.  Which means that adding $\sigma_1$ to $\hat B$ removes a branching point of type 
$(r , m)$ and instead adds two branching points - one of type $(r_{\sigma_1}, m)$ and one of type 
$(r , r_{\sigma_1})$. 
Therefore,
\begin{align} \label{eq:Bhat-B2}
    \Phi_n (B) &= \Phi_n  (\hat B) + \lb h_{r, r_{\sigma_1}} - h_{r, m}\rb + h_{r_{\sigma_1} , m} 
    \,\geq\, \Phi_n  (\hat B) + h_{r_{\sigma_1} , m} \\
    &\geq \Phi_n  (\hat A) + h_{r_1 , m_1} 
    \,=\, \Phi_n (A)    
\end{align}
\qed

Let us turn to verifying \eqref{eq:MC-b}. 
Let $A, B$ be two disjoint subsets of $\bbL_0$.

\noindent
Consider first $\calB (A)\cap \calB (B) = \emptyset$. 
Then there exists unique $z$ such that
\[
\lbr {z}\rbr = \argmin\lbr\abs{u\wedge v }~:~u\in A ,v\in B\rbr .
\]
Therefore, $\calB (A\cup B ) = \calB (A )\cup \calB (B ) \cup \lbr z\rbr$. 
For each $u\in\calB (A\cup B )$, let $\fra_{A\cup B} (u )$ be its branching ancestor.
\begin{rem} 
\label{rem:fra-u}
Recall that by our convention there is always a branching  
ancestor in some  generation $\ell \leq n+1$. In the sequel we shall,  with a certain abuse of 
notation, use $\abs{\fra_A (u )}$ for the generation of  $\fra_A (u )$. In this way it might happen that 
$\abs{\fra_A (u )} = n+1$. 
\end{rem}
Now, we can rewrite \eqref{eq:BP-ex2} as follows
\be{eq:Phi-AB} 
\Phi_n (A\cup B ) = \sum_{u \in (A \cup B) \setminus \lbr z \rbr } h_{\abs{\fra_{A\cup B} (u )}, \abs{u}} + 
h_{\abs{\fra_{A\cup B} (z )}, \abs{z}} + h .
\ee 
For any $u\in \calB (A)$, $\abs{\fra_{A} (u)} \geq \abs{\fra_{A\cup B} (u )}$. 
The same regarding $u\in \calB (B)$. 
Hence, the first term on the right hand side of \eqref{eq:Phi-AB} is less or equal to 
\[ 
 \sum_{u\in \calB (A)} h_{ \abs{\fra_{A} (u )},  \abs{u}} +  
 \sum_{u\in \calB (B)} h_{\abs{\fra_{B} (u )},  \abs{u}}. 
\]
Since, by construction, $h_{ \abs{\fra_{A\cup B} (z )}, \abs{z}} \leq h$, the assertion  
$\Phi (A\cup B ) \leq \Phi (A ) +\Phi (B )$ follows. 
\smallskip

If $\abs{\calB (A )\cap \calB (B)} = N >0$, then the following happens: 
\begin{itemize}
    \item As before for any $u\in \calB (A )$, it holds that $h_{\abs{\fra_A (u )} ,\abs{u}} \geq h_{\abs{\fra_{A\cup B} (u )} ,\abs{u}}$. The same regarding $u\in \calB (B)$.
    \item However, each one of $N$ branching points $w\in \calB_{AB}:=\calB (A )\cap \calB (B )$ is counted twice in $\Phi_n (A ) + \Phi_n (B )$, whereas it is counted once in $\Phi_n (A\cup B)$.
    \item On the other hand there are exactly $N+1$ new branching points $z\in \calB_{A\Delta B} := \calB (A\cup B) \setminus \lb \calB (A )\cup \calB (B)\rb$. 
\end{itemize}
Therefore we just need to find a matching $\tau$ between any $N$ out of the $(N+1)$ points of $\calB_{A\Delta B}$ and $N$ points of $\calB_{AB}$ such that, 
\be{eq:tau-AB} h_{\abs{\fra_A (\tau (z) )}\wedge \abs{\fra_B \tau (z))} , \abs{\tau (z)}} \geq h_{\abs{\fra_{A\cup B} (z)} , \abs{z}}
\ee
for any one of these $N$ chosen points $z\in \calB_{A\Delta B}$. Indeed, given such a $\tau$ 
the remaining $(N+1)^{\text{\ul{st}}}$ point $w$ of $\calB_{A\Delta B}$ will always satisfy $h_{\abs{\fra_{A\cup B} (w)}, \abs{w}} \leq h$, and hence 
\eqref{eq:tau-AB} indeed implies \eqref{eq:MC-b} in its full generality. The mapping $\tau$ defined in the previous subsection applies here as well. For all $z \in \calB_{A \Delta C}, z \neq w$
\begin{align*}
    |\fra_{A \cup B} (z)| \leq |\tau (z)| < |\fra_{A \cup B} ( \tau (z) )| \leq | \fra_{A} ( \tau (z) ) | \wedge | \fra_{B} ( \tau (z) ) |
\end{align*}
This completes the proof.
\end{proof}

\subsection{Existence of free energy} 
\begin{proof}[Proof of proposition~\ref{prop:B2-free-energy-exist}]
    Let $n \in \bbN$ and $A \subseteq \bbL_{0, r/l}^{(n)}$.
    Denote by $\ell^*$ the maximal height of a branching point in the tree induced by $A$.
    Then for $A \neq \emptyset$, by monotonicity of $\{h_{k,\ell}\}$,
    \begin{align*}
        \Phi_n (A) - \Phi_{n-1} (A) & = h_{n+1, n} + h_{n+1, \ell^*} - h_{n, n-1} - h_{n, \ell^*} \\
        & \leq 2h_{n+1, n} \,.
    \end{align*}
    Set $\gamma_n = 2 h_{n+1, n}$ then $\sum_{n} \gamma_n 2^{-n} < \infty$ and assumption~\eqref{eq:Phi-A} holds.
    By Proposition \ref{prop:Phi-FE-A} we get the desired conclusion.
\end{proof}

\subsection{Sufficient condition for wetting}
As in the case of first order branching, for the sake of simplifying the proofs to follow, we shall assume henceforth that $h = 0$ instead of $h_{n+1, n}$ in the definition of $\Phi_n$ as a second order branching clustering function. This will again not change the canonical free energy as $2^{-n} h_{n+1,n} \to 0$ as $n \to \infty$ thanks to the assumption~\eqref{eq:B2-free-energy-exist} which will be assumed throughout. 

Also, as before, if $\ul b$ is the second order branching pattern for the subtree $\cT(A)$ induced by $A \subseteq \bbL_0^{(n)}$ in $\bbT^{(n)}$, we write $A \sim \ul b$. 
For $\ell \in [0,n]$ we also let $a_\ell$ be the total population size at generation $\ell$ or higher, so that 
\begin{equation}
\label{e:4.10}
a_\ell : = 1 + \sum_{j = \ell}^n b_j \;, \quad \ell = 1, \dots, n
\; , \quad
a_0 = b_0 = |A| \,,
\end{equation}
where we recall that $(b_k)_{k=0}^n$ is the first order branching pattern, which is related to the second order branching pattern via,
\begin{equation}
\label{e:4.11}
2b_k = \sum_{\ell=0}^{k-1} b_{k,\ell} 
\;
,\quad
b_\ell = \sum_{k = \ell+1}^{n+1} b_{k,\ell}  \,,
\end{equation}
for all $\ell,k \in [0,n]$.

Given $n \geq 1$,  $a_0 \in [0, 2^n]$, we shall say that a triangular array $\ul b = (b_{k,\ell})_{0\leq \ell < k \leq n+1}$ is an {\em admissible second order branching pattern for} $a_0$ and write $\ul b \sim a_0$ if there exists a subset $A \subseteq \bbL^{(n)}_0$ with $|A| = a_0$, such that $A \sim \ul b$. 
It is an {\em admissible second order branching pattern} if it is admissible for some $a_0$. 
The number of subsets of leaves $A \subseteq \bbL_0^{(n)}$ with a given (admissible) second order branching pattern $\ul b$ is given by
\begin{equation}
\ent{\ul b} = \# \{ A \subseteq \bbL^{(n)}_0 \,:\, A \sim \ul b \}\,.
\end{equation}
As in the first order branching case, we can write the canonical partition function as
\be{eq:L1-Gpf-rep} 
\bbW_n (a_0 ) := 
\sum_{\ul b \sim a_0} \ent{\ul b} {\rm e}^{-\Phi_n (\ul b)}\,,
\ee
where 
\begin{equation}
\Phi_n(\ul b) = \sum_{k=1}^{n+2} \sum_{\ell=0}^{k-1} h_{k, \ell} b_{k, \ell} \,,
\end{equation}
with the convention that $b_{n+2,\ell} = \1_{\lbr \ell = n+1 \rbr}$.

\begin{lem}\label{lem_entropy_bs}
For $n \geq 1$, if $\ul b = (b_{k, \ell} )_{0 \leq \ell < k \leq n+1}$ is an admissible second order branching pattern, then
\begin{equation}\label{two_layer_entropy}
\ent{\ul b} = \prod_{j= 1}^{n} \binom{2 b_j}{b_{j,j-1}, \ldots, b_{j,0}  } 2^{ a_j - b_j},
\end{equation}
above $\binom{2 b_j}{b_{j,j-1}, \ldots, b_{j,0}  }:= b_j! \Big/ b_{j,j-1} ! \dots b_{j,1} ! b_{j,0}!$ is the multinomial coefficient.
\end{lem}
\begin{proof}
Notice that two sets $A$ and $A'$ that are compatible with $\ul b$ may be different in a graph sense, nevertheless they will always have the same population size in every generation.  
Each non branching point has exactly one offspring that can be embedded in $\bbT_n$ in two distinct ways (``right" or ``left"). 
Therefore, after fixing the branching points in generation $j$ there are $2^{ a_j - b_j}$ different ways to obtain the next vertex in generation $j-1$, giving the $2^{ a_j - b_j}$ term in \eqref{two_layer_entropy}. 
Finally the first term in the product is the number of possible ways one may arrange the branching points $b_{j+i,j}$ in generation $j+i$ whose last branching ancestor is in generation $j$. 
\end{proof}

The next lemma shows that the free energy does not change if one replaces 
the sum in~\eqref{eq:L1-Gpf-rep} by the maximum.
\begin{lem}\label{pqrt_fct_bp_max} 
For all dyadic $\epsilon > 0$, 
\begin{equation}\label{eq_pqrt_fct_bp_max}
\omega (\eps) = \lim_{n \to \infty} \frac{1}{2^n} \ln 
\max_{\ul b \sim \epsilon 2^n} \big( \ent{\ul b} {\rm e}^{-\Phi_n (\ul b)}\big) \,.
\end{equation}
\end{lem}
\begin{proof}
Clearly,
\begin{equation}
\label{e:4.18}
\max_{\ul b \sim \epsilon 2^n} \big( \ent{\ul b} {\rm e}^{-\Phi_n (\ul b)}\big)
\leq\, \bbW_0(\epsilon 2^n)\, \leq \#\{\ul b :\: \ul b \sim \epsilon2^n \}
\max_{\ul b \sim \epsilon 2^n} \big(\ent{\ul b} {\rm e}^{-\Phi_n (\ul b)}\big)
\end{equation}
Since $\sum_{0 \leq \ell < k \leq n+1} b_{k,\ell} = 2a_0 - 1 = 2\epsilon 2^n-1$,  the total number of admissible arrays $\ul b \sim \epsilon 2^n$ is smaller than $(2\epsilon 2^n-1)^{n^2} \leq \rme^{C n^3}$. Therefore, the limit as $n \to \infty$ of $2^{-n}$ times the logarithm of both bounds in~\eqref{e:4.18} tend to the right hand side of~\eqref{eq_pqrt_fct_bp_max}.
\end{proof}

For all $k \in \bbN$ and $0 \le \ell \le k - 1$, set
    \begin{align} \label{eq:b2-p-def}
        p_{k, \ell} = h_{k, \ell} - (\ln 2)\ell \,.
    \end{align}

\begin{proof}[Proof of Proposition~\ref{p:1.13}]
Let $\epsilon > 0$, $n \geq 1$ and take $\ul b \sim \epsilon 2^n$ to be the maximizer in~\eqref{eq_pqrt_fct_bp_max}. As in~\eqref{e:3.10} we can wrtie
\begin{equation}
\begin{split}
    \sum\limits_{j=1}^{n}(a_j - b_j) &= (n+2) - 2 a_0 + \sum\limits_{\ell=0}^{n} \ell b_\ell \\
    & = (n+2) - 2 a_0 + \sum\limits_{\ell=0}^{n} \sum\limits_{k=\ell+1}^{n+1} \ell b_{k,\ell} 
    = (n+2) - 2a_0 + \sum\limits_{k=1}^{n+1} \sum\limits_{\ell=0}^{k-1} \ell b_{k,\ell}
\end{split}
\end{equation}
Thanks to Lemma~\ref{lem_entropy_bs} and Lemma~\ref{pqrt_fct_bp_max} we thus get
\begin{equation}
\begin{split}
\label{e:4.19}
    \omega(\epsilon) & = - \epsilon (2 \log 2) + \lim_{n \to \infty} 2^{-n} \log \left( \left[ \prod\limits_{k=1}^{n} \binom{2 b_{k}}{b_{k,k-1}, \ldots, b_{k,0}  }  \right] \times \prod\limits_{k=1}^{n+2} \prod\limits_{\ell=0}^{k-1} e^{- p_{k, \ell} b_{k, \ell}} \right)\\
  & = - \epsilon (2 \log 2) + \lim_{n \to \infty}
       2^{-n} \sum_{k=1}^{n} \log \left[ \binom{2 b_{k}}{b_{k,k-1}, \ldots, b_{k,0}  } \prod_{\ell=0}^{k-1} e^{-p_{k, \ell} b_{k, \ell}} \right] \,.
\end{split}
\end{equation}
where the last equality follows since 
\begin{equation}
\begin{split}	
\sum_{k=n+1}^{n+2} & \sum_{\ell=0}^{k-1} p_{k,\ell}b_{k,\ell} \le \max_{k=n+1,n+1} \max_{\ell=0,\dots,k-1} \left| p_{k,\ell} \right| \sum_{k=n+1}^{n+2} \sum_{\ell=0}^{k-1} b_{k,\ell}  \\
    & = \frac{1}{2} \left( b_{n+1} + b_{n+2} \right) \max_{k=n+1,n+2} \max_{\ell=0,\dots,k-1} \left| p_{k,\ell} \right| 
    \leq \left( h_{n + 2, n + 1} + \lb n + 1 \rb \ln{2} \right)\,.
\end{split}
\end{equation}
By Stirling's approximation
\begin{equation}
\log \binom{2 b_k}{b_{k,k-1} \ \ldots \ b_{k,0}  } = 2 b_k \log 2 b_k - \sum_{l = 0}^{k-1} \Big( b_{k,l} \log b_{k,l} 
\Big) + O(n^2) \,,
\end{equation}
so the last sum in~\eqref{e:4.19} is equal to $O(n^3)$ plus
\begin{equation}
\label{e:4.23}
\begin{split}
-\sum_{k=1}^{n} \sum_{\ell = 0}^{k-1} b_{k,\ell} \log \Big( \frac{b_{k,\ell}}{ 2 b_{k} e^{-p_{k, \ell}} } \Big) 
& =- \sum_{k=1}^{n} 2 b_{k} \left[
 \sum_{\ell = 0}^{k-1} e^{-p_{k, \ell}} \Big( \frac{b_{k,\ell} }{ 2b_{k} e^{-p_{k, \ell}} } \Big) \log \Big( \frac{b_{k,\ell} }{ 2b_{k} e^{-p_{k, \ell}} } \Big) \right] \\
& \leq e^{-1} \sum_{k=1}^{n} 2 b_{k} 
 \sum_{\ell = 0}^{k-1} e^{-p_{k, \ell}} \leq 2e^{-1} \kappa_2 \epsilon 2^n \,.
\end{split}
\end{equation}
Above, the first inequality follows since $y \log y \geq -e^{-1}$ for all $y > 0$, and $\kappa_2$ is as in~\eqref{e:1.36}.
Plugging this in~\eqref{e:4.19}, taking the limit as $\epsilon \to 0$ and  using Theorem~\ref{thm:Phi-wet} we complete the proof.
\end{proof}

\subsection{Necessary condition for wetting} 
Recall the definition of $g_{\ell,d}$ from~\eqref{eq:b2-g-def}.
\begin{lem} \label{lem: necessary b2 lemma}
Let $(h_{k,j})_{0 \leq j < k}$ be non-decreasing triangular array and suppose that \linebreak \eqref{eq:B2-free-energy-exist} holds. Suppose also that $g_{\ell,d} \geq 0$ for all $\ell,d \geq 0$. 
Then there exists arbitrarily small $s$ such that $e^s - 1$ is dyadic and
\begin{equation}
-J^* \geq \ln{\frac{1}{s}} - \big(2s^2 + O(s^3)\big) \hat{g}(s,\, 2s) - C \,.
\end{equation}
\end{lem}
\begin{proof}
We proceed as in the case of first order branching. If $\wh{g}(s,2s) = \infty$ for some $s > 0$, then the statement holds trivially, so we shall assume otherwise. 
Let $(\epsilon_j)_{j \geq 1}$ be a sequence which tends to $0$ as $j \to \infty$. 
For each $j \geq 1$ and all $n$ large enough, let $\ul b^{(j,n)} = (b^{(j,n)}_{k, \ell})_{0 \leq \ell < k \leq n}^n$ be an admissible second order branching pattern on $\bbT^{(n)}$ for $a_0^{(j,n)} = \epsilon_j 2^n$. 
It follows from Theorem~\ref{thm:Phi-wet} and~\eqref{e:4.19} that
\begin{equation}
\label{e:4.24}
    -J^* \geq \lim_{j \to \infty} \limsup_{n \to \infty}\, B_n(\ul b^{(j,n)}) - C \,,
\end{equation}
where
\begin{equation}
\label{e:4.25}
B_n(b) = a_0^{-1} \sum_{k=1}^{n} \ln \left[ \binom{2 b_{k}}{b_{k,k-1}, \ldots, b_{k,0}  } \prod_{d=1}^{k} e^{- g_{k - d, d} b_{k, k - d}} \right] 
\end{equation}
Now, for $t > 0$ dyadic, $j \in \bbN$ and all $n \geq j$ set,
\begin{equation}
\label{eq:B2-Bernoulli-branching-def}
        t_k :=  \begin{cases}
            1 & j < k < n  \,, \\
            t & 0 \leq k \leq j \,,
        \end{cases}
\end{equation}
let $(a_k)_{k=0}^n$, $(b_k)_{k=0}^n$ be defined via~\eqref{e:4.10}, 
\begin{equation}
	\frac{b_k}{a_k} = t_k \ , \ \ 1 \leq k \leq n \,,
\end{equation}
and set
\begin{equation}
\label{eq:B2-branching-pattern}
        b_{k, k-d} = \begin{cases}
            2 b_k \cdot t_{k - 1} \cdot (1 - t_{k - 1})^{d - 1} & 1 \leq d < k \leq n \,,\\
            2 b_k \cdot (1 - t_{k - 1})^{k-1} & d = k \in [1, n] \,.
        \end{cases} 
\end{equation}
We claim that for all such $t$ and $j$ whenever $n$ is large enough the triangular array
$b = (b_{k, \ell})_{0 \leq \ell < k \leq n + 1}$ is an admissible second order branching pattern for $a_0 = 
\eps_{j} 2^n$ with $\eps_j = ((1 + t)/2)^j$.
To see this, we need to verify that the $b_{k,\ell}$-s are all integer and that~\eqref{e:4.11} holds. 
Indeed, 
\begin{equation}
b_k = 
\begin{cases}
	2^{n-k}	& j+1 \leq k \leq n\,, \\
	 2^{n-j}(1+t)^{j-k}t  & 1 \leq k \leq j\,,
\end{cases}
\end{equation}
and
\begin{equation}
\label{e:4.30}
        b_{k, k-d} = \begin{cases}
            2b_k1_{\{d=1\}} & k \geq j+1 \,,\\
            2 b_k t (1 - t)^{d - 1} & 1 \leq d < k \leq j \,,\\
            2 b_k (1 - t)^{k-1} & d = k \leq j \,.
        \end{cases} 
\end{equation}
The right hand side above can be made integer for all dyadic $t > 0$ and $j \in \bbN$ by taking $n$ large enough.
As for~\eqref{e:4.11}, the first equality in~\eqref{e:4.11} follows easily by summing over $d=1, \dots, k$ in~\eqref{e:4.30}. 
The second trivially holds for $\ell \geq j+1$. 
For $1 \leq \ell \leq j$, we have
 \begin{equation}
 \begin{split}	
        \sum_{k = \ell + 1}^{n + 1} b_{k, \ell} 
        & = \sum_{d = 1}^{j  - \ell} 2 b_{\ell + d} t (1 - t)^{d - 1} 
        = \sum_{d = 1}^{j - \ell} 2 \cdot 2^{n - j} t (1 + t)^{j - \ell - d} \cdot t (1 - t)^{d - 1} \\
        & = 2 t^2 \cdot 2^{n - j} (1 + t)^{j - \ell - 1} \cdot \frac{1 - \left( \frac{1 - t}{1 + t} \right)^{j - \ell}}{1 - \frac{1 - t}{1 + t}} 
        = 2^{n - j} \cdot t (1 + t)^{j - \ell} 
        = b_\ell\,,
\end{split}
\end{equation}
Turning to $B_n(\ul b)$, for the sum in~\eqref{e:4.25} when $k \in [j+1, n]$ we have,
\begin{equation}
a_0^{-1} \sum_{k = j + 1}^{n} 2 b_k g_{k - 1, 1} = \left( \frac{2}{1 + t} \right)^{j} 2^{-n} \sum_{k = j + 1}^n 2 \cdot 2^{n - k} g_{k - 1, 1} 
        \leq \sum_{\ell = j}^{n - 1} (1+t)^{-\ell} g_{\ell, 1} \,,
\end{equation}
since $\hat{g}(s_1,s_2)$ is finite for all $(s_1, s_2) \in (0,\infty)^2$, the last sum tends $0$ as $n \to \infty$ followed by $j \to \infty$.

At the same time, as in~\eqref{e:4.23}, the sum in~\eqref{e:4.25} taken over $k \leq j$, is equal to $O(n^3)$ plus
\begin{equation}
\label{e:4.33}
\begin{split}
- 2 \sum_{k=1}^j b_k \sum_{d=1}^{k} & \frac{b_{k, k - d}}{2 b_k} \ln{\frac{b_{k, k - d}}{2 b_k e^{-g_{k - d, d}}}} \\
& = -2 \sum_{k = 1}^j b_k \sum_{d = 1}^{k} t (1 - t)^{d - 1} \left( (d - 1) \ln{(1 - t)} + \ln{t} + g_{k - d, d} \right) \\
        &\ge -2 \sum_{k = 1}^{j} b_k \sum_{d = 1}^{k} t (1 - t)^{d - 1} \left( \ln{t} + g_{k - d, d} \right) \\
        &= -2 a_0 t^2 \sum_{\ell = 0}^{j - 1} \sum_{d = 1}^{j - \ell} (1 + t)^{- \ell - d} (1 - t)^{d - 1} \left( \ln{t} + g_{\ell, d} \right) \,,
 \end{split}
 \end{equation}
 where for the last equality we plugged in $b_k = t a_k = a_0 t(1 + t)^{-k}$ and substituted $\ell = k-d$.
Thanks to the assumed positivity of $g$, taking $j \to \infty$ amounts to turning the last double sum into a double series with both indices going to $\infty$. This gives
\begin{equation}
\frac{1}{1-t} \sum_{\ell = 0}^{\infty} \sum_{d = 1}^{\infty} (1 + t)^{- \ell} \lb \frac{1 + t}{1 - t} \rb^{-d} (\ln t + g_{\ell, d}) 
\leq \frac{1}{2t^2} \ln t +
\sum_{\ell, d = 0}^{\infty} (1 + t)^{- \ell} \lb 1+t \rb^{-2d} g_{\ell, d} \,,
\end{equation}
where we have used the explicit formula for the geometric series and the positivity of $g$.
Collecting all estimates, the double limit in~\eqref{e:4.24} is at least
\begin{equation}
\ln{\frac{1}{t}} - 2 \big(t^2 + O(t^3)\big) \hat{g}\big(\ln (1+t),\, 2\ln (1+t) \big) - C \,.
\end{equation}
Finally, substituting $t= e^s - 1 = s+O(s^2)$ and noting that $\ln t^{-1} = \ln s^{-1} + O(s)$, we get the desired statement.
\end{proof}

We can now give
\begin{proof}[Proof of Proposition~\ref{prop:B2-general-necessary}]
By replacing $g_{\ell, d}$ by $g^{+}_{\ell, d} = g_{\ell, d} \vee 0 \ge g_{\ell, d}$, which only decreases the free energy, we may assume w.l.o.g. that $g_{\ell, d} \geq 0$ for all $\ell$ and $d$. Then
\eqref{eq:B2-general-assumption-for-no-wetting} implies that
$\hat{g}(s, 2s) = O (s^{-2} \log (1/s))$ so that $\hat{g}(s, 2s) O(s^3) = o(1)$ as $s \to 0$. In view of Lemma~\ref{lem: necessary b2 lemma} we thus get $J^* = -\infty$ and therefore there is no wetting transition.
\end{proof}

\begin{proof} [Proof of Corollary~\ref{cor: b2 additive necessary}]
Define $g^{(1)}_l := h^{(1)}_l - (\ln 2) l$, so that 
    \begin{align} \label{eq:B2-necessary-additive-structure}
        g_{\ell, d} = g^{\lb 1 \rb}_\ell + h^{\lb 2 \rb}_d \,.
    \end{align}
By assumption $|g^{(1)}_k|, |h_k^{(2)}| = O(k)$ as $k \to \infty$, so that
$|\hat{g}^{(1)}(s)|, |\hat{h}^{(2)}(s)| = O(s^{-2})$ as $s \downarrow 0$.
The above additive structure then gives
\begin{equation}
\label{e:4.34}
\begin{split}	
        \hat{g} \lb s, 2s \rb & = \sum_{\ell = 0}^\infty \sum_{d = 0}^\infty e^{-s \ell - 2s d} \lb g^{\lb 1 \rb}_\ell + h^{\lb 2 \rb}_d \rb \\
        & = \sum_{\ell = 0}^\infty e^{-s\ell} g^{\lb 1 \rb}_\ell \sum_{d = 0}^\infty \lb e^{-2s} \rb^d + \sum_{d = 0}^\infty e^{-2s d} h^{\lb 2 \rb}_d \sum_{\ell = 0}^\infty \lb e^{-s} \rb^\ell \\
        & =  \lb \frac{1}{2s} \hat{g}^{\lb 1 \rb} \lb s \rb + \frac{1}{s} \hat{h}^{\lb 2 \rb} \lb 2s \rb \rb \lb 1 + O \lb s \rb \rb \\ &
        = \frac{1}{2s} \hat{g}^{\lb 1 \rb} \lb s \rb + \frac{1}{s} \hat{h}^{\lb 2 \rb} \lb 2s \rb + O(s^{-2})\,.
\end{split}
\end{equation}
At the same time,
    \begin{align}
    \label{e:4.35}
        \hat{h}_{\lfloor \cdot/2 \rfloor}^{\lb 2 \rb} (s) &= \sum_{d = 0}^\infty e^{-ds} h_{\lfloor \frac{d}{2} \rfloor}^{\lb 2 \rb} = (2 + O(s)) \sum_{d = 0}^\infty e^{-2ds} h_d^{\lb 2 \rb}  = 2\hat{h}^{\lb 2 \rb} (2s) + O(s^{-1}) \,.
    \end{align}

Now, as in the proof of Corollary~\ref{c:1.11}, condition~\eqref{eq:B2-additive-necessary} implies that
    \begin{align} \label{eq: B2 4.39}
        \lim_{s \downarrow 0} \ln{\frac{1}{s}} - s \hat{g}^{\lb 1 \rb} \lb s \rb - s \hat{h}_{\lfloor \cdot/2 \rfloor}^{\lb 2 \rb} \lb s \rb = \infty \,,
    \end{align}
where we have used the linearity of the Laplace Transform.
Combining~\eqref{e:4.34},~\eqref{e:4.35} and~\eqref{eq: B2 4.39}, this gives
    \begin{align*}
        \lim_{s \downarrow 0} \lb \ln{ \frac{1}{s}} - 2 s^2 \hat{g} \lb s, 2s \rb \rb = \infty
    \end{align*}
    so there is no wetting transition.
\end{proof}

\section{Clustering via capacity}
\label{s:4}
\subsection{Properness of definition}
\begin{proof}[Proof of Proposition~\ref{p:1.5}]
The proof is based on the following argument by Lyons \cite{Lyons90, Lyons92}: 
Let $f$ be the minimizer ($\bbT^{(n)}$ is a finite graph - no ambiguity) 
in \eqref{eq:Cap}. Then 
\be{eq:theta-flow} 
\theta_A (e ):= \frac{1}{{\rm CAP}_\calC (A )} \sfC_e \nabla_e f
\ee
is a unit flow from $\mathsf{0}$ to $\bbL_0$.
Define resistances $\sfR_e = \sfC_e^{-1}$. Given a flow $\theta$ define
its energy $\calE (\theta )= \sum_e \sfR_e \theta (e )^2$. Then,  $\calE (\theta_A ) = 
{\rm CAP}_\calC (A )^{-1}$, and  $\theta_A$ has the minimal energy within the family 
$\calF_{A}$ 
of
unit flows from $\mathsf{0}$ to $A$.

Next,  any unit flow $\theta\in \calF_A$ could be viewed 
as a probability distribution $\mu$ on $A\subseteq \bbL_0^{(n)}$; $\mu (u ) = \theta (e (u ))$. 
Kirchoff's law implies that for any 
 $u\not\in \lbr  \mathsf{0}\rbr\cup \bbL_0$, 
$
\theta (e(u )) = \sum_{e^\prime} \theta (e^\prime) , 
$ 
where the summation is with respect all outgoing edges from $u$ in the direction of  $\bbL_0$. Consequently, if $\calD (u )\subseteq \bbL_0$
 is the set of all the descendants of $u$ in $\bbL_0$, then 
 $
 \theta (e (u ))= \sum_{v\in \calD (u )} \mu (v) .
 $ 
 
 Recall that $R_e = R_{\abs{e}}$, and set 
\be{eq:rho-R}
\alpha_\ell^{(n)} = \1_{\ell <n}\sum_{j=\ell}^n \sfR_j. 
\ee 
Then, summing by parts in  
\[ 
\calE (\theta ) = \sum_{\abs{u} <n } R_{\abs{u}}
\lb \sum_{v\in \calD (u )}\mu (v )\rb^2 
\]
  one recovers the following formula:
\be{eq:E-Lyons} 
\calE (\theta )  = 
\sum_{u, v} \mu (u )\mu (v ) \alpha_{\abs{ u\wedge v}}^{(n)}  \ \text{ and } \ 
{\rm CAP}_\calC (A ) = \max_{\mu (A) = 1}  
\lbr \sum_{u,  v} \mu (u )\mu (v ) \alpha_{ \abs{ u\wedge v }}^{(n)}\rbr^{-1}
\ee
At this point, 
since $\alpha^{(n)}$ in \eqref{eq:rho-R} is non-increasing, 
we recover \eqref{eq:MC-a}, that is 
 $A\prec B$  implies  
that  ${\rm CAP}_\calC (A )\leq {\rm CAP}_\calC (B )$. 
\end{proof}

\subsection{No wetting transition}
\begin{proof} 
Since we know that ${\rm CAP}_\calC $ is a monotone clustering function, 
\[ 
{\rm CAP}_\calC  (A) \leq \sum_{u\in A } 
{\rm CAP}_\calC \lb\{ u\}\rb 
\stackrel{\eqref{eq:CAP-A-1}}{=} 
\sum_{u\in A}\sfC_0 \bbP_u^{\sf RW} \lb \tau_0 <\tau_u\rb \leq \sfC_0 \abs{A}. 
\]
It follows that $\sfZ^{\Phi,-J}_n$ in \eqref{eq:Phi-J} with $\Phi(A) = {\rm CAP}_\calC  (A)$,
satisfies
\[
\sfZ_n^{\Phi,-J} \geq \sum_{N=1}^{2^n} \binom{2^n}{N} \, 
{\rm e}^{-N (J +\sfC_0 )}. 
\]
Since, 
\[ 
\lim_{m_0\to\infty}	\lim_{n\to\infty}
\frac{1}{2^{n-m_0}} \ln \binom{2^n}{2^{n-m_0}}
=\infty , 
\]
the result follows from Corollary~\ref{c:1.9}.
\end{proof}

\section*{Acknowledgments}
The research was supported by ISF grants no.~2870/21 and~1723/14, and also by the BSF award 2018330. 

\section*{Data Availability}
No datasets were generated or analyzed during the current study.

\section*{Conflict of Interest}
The authors have no relevant financial or non-financial interests to disclose.

\bibliographystyle{abbrv}
\bibliography{pinning}

\end{document}